\documentclass[11pt,leqno]{article}
 \usepackage{epsfig}
\usepackage{amssymb}
\usepackage{color}
\usepackage{mathrsfs}
\usepackage{amsmath,amsthm,amscd}

\newcommand\N{{\mathbb N}}
\newcommand\R{{\mathbb R}}

\def\AA{{\mathcal A}}
\def\BB{{\mathcal B}}
\def\CC{{\mathcal C}}

\def\HH{{\mathcal H}}

\def\KK{{\mathcal K}}
\def\LL{{\mathcal L}}

\def\SS{{\mathcal S}}

\def\eps{{\varepsilon}}

\newcommand{\lp}{\left(}
\newcommand{\rp}{\right)}

\newtheorem{theo}{Theorem}[section]
\newtheorem{prop}[theo]{Proposition}
\newtheorem{lemma}[theo]{Lemma}
\newtheorem{cor}[theo]{Corollary}
\newtheorem{rem}[theo]{Remark}
\newtheorem{conjecture}[theo]{Conjecture}

\newtheorem{defin}[theo]{Definition}

\numberwithin{equation}{section}

\newcommand{\uO}{u_{\Omega}}
\newcommand{\dd} {\textrm{dist}}
\newcommand{\divergence}{\textrm{div}}
\newcommand{\Per}{\textrm{Per}}

\newcommand{\beqn}{\begin{equation}}
\newcommand{\eeqn}{\end{equation}}
\newcommand{\bear}{\begin{eqnarray}}
\newcommand{\eear}{\end{eqnarray}}
\newcommand{\bean}{\begin{eqnarray*}}
\newcommand{\eean}{\end{eqnarray*}}

\date{\today}

\begin{document}

\title{A discrete Bernoulli free boundary problem}

\author{Maria del Mar Gonzalez\footnote{Universitat Polit\`ecnica de Catalunya, ETSEIB - Dept. de MA1, Av. Diagonal 647, 08028 Barcelona,Spain}, 
Maria Gualdani\footnote{Department of Mathematics, University of Texas at Austin, 2515 Speedway C1200 , Austin TX, 78712, USA },
Henrik Shahgholian\footnote{Department of Mathematics, The Royal Institute of Technology 100 44 Stockholm, Sweden}}

\maketitle

\begin{abstract}
We consider a free boundary problem for the $p$-Laplace operator which is related to the so-called Bernoulli free boundary problem. In this formulation, the classical boundary gradient condition is replaced by a condition on the distance between two different level surfaces of the solution. For suitable scalings our model converges to the classical Bernoulli problem; one of the advantages in this new formulation lies in the simplicity of the arguments, since one does not need to consider the boundary gradient.

We shall study this problem in convex and other regimes, and establish existence and qualitative theory.

\end{abstract}



\section{Introduction}
\label{sec:intro}
\setcounter{equation}{0}
\setcounter{theo}{0}

For a bounded domain $K\subset \R^N$ ($N\geq 2 $), the well known (exterior) Bernoulli free boundary problem is to find a domain $\Omega \supset \overline K$, with the requirement that the harmonic function $u$ in the region $\Omega\setminus \overline K$ with given boundary values satisfies a prescribed Neumann condition on $\partial \Omega$;
the boundary $\partial \Omega$ is thus called Bernoulli free boundary. There is a vast literature on the subject, including by-now classical references
 \cite{Acker:heat-flow,Acker:interior,Acker-Meyer:free-boundary-problem, Alt-Caffarelli:minimum-problem, Beurling89, Flucher-Rumpf:Bernoulli}. One may also replace the Laplacian operator by the $p$-Laplacian operator;
 this problem arises in various nonlinear flow laws, and several physical situations,
 e.g. electrochemical machining and potential flow in fluid mechanics (see for instance
  \cite{Acker-Meyer:free-boundary-problem,DP05,GrecoKawohl09,HH00,HH000,P08,Atkinson-Champion}).

In this paper we propose a new formulation of the $p$-Laplacian Bernoulli problem for $1<p<+\infty$, in which the boundary gradient condition is replaced by
 a (weaker) condition on the distance between two different level sets of the solution. More precisely, given the $p$-capacitary potential $u$ in the annulus $\Omega\setminus \overline K$ we ask that
\begin{equation}\label{main-condition}\dd (x,\{u=l\})=\lambda, \quad \forall x\in\partial\Omega,\end{equation}
where $\{u=l\}$ is the $l$-level set of $u$ and $l,\lambda$ are two positive given constants.

Condition \eqref{main-condition} was considered in \cite{Magnanini-Sakaguchi:Matzoh,Magnanini-Sakaguchi:unbounded,Magnanini-Sakaguchi:stationary-surface}, while studying diffusion problems with stationary isothermic surfaces; and independently by the third author in \cite{Shahgholian:Serrin}, where he considered a discrete version of Serrin's problem. This second approach was inspired by a parabolic free boundary model in finance in which the distance between the free boundary (that is the zero level set for the solution)
and the location of a source  is prescribed (c.f. \cite{Lasry-Lions, price-formation2, Chayes-Gonzalez-Gualdani-Kim, Caffarelli-Markowich-Pietschmann:price-formation}
for an overview).

The novelty in the approach introduced here is that our problem
 with condition \eqref{main-condition} approximates the solution of the
classical Bernoulli problem in the following sense: as $\lambda \to
0$, any $p$-capacitary function $u_\lambda$ vanishing on $\partial
\Omega$ and satisfying \eqref{main-condition} with $ l = \omega
\lambda$ converges to a $p$-capacitary function $u$ vanishing on
$\partial \Omega$ such that $|\nabla u|=\omega$ on $\partial
\Omega$. In some sense, this new formulation can be understood as a
\emph{discrete Bernoulli problem}.

Two different problems can be set up according to where the free boundary lies in the annulus: the exterior problem $(P_E)$, considered in Section \ref{section-exterior},
and the interior one $(P_I)$, considered in Section \ref{section-interior}.  It is clear that the geometry of the domains will play a crucial role in both problems.

In this paper we deal with two different geometries: the convex and the star-shaped case. We address
the questions of existence, multiplicity, uniqueness and regularity for the discrete Bernoulli problem, together with convergence to the usual Bernoulli problem in the convex setting.

The main ideas in our paper follow the approach by sub- and super-solutions introduced by Beurling in \cite{Beurling89}.
For the classical Bernoulli problem in the convex case with $p$-Laplacian operator, questions of existence and uniqueness were addressed in \cite{HH00,HH000,HH97}. It is noteworthy that the use of the properties of the distance function \eqref{main-condition} avoids many of the technicalities: for example, higher regularity for solutions is an immediate consequence from existence theory, which is not the case for the classical Bernoulli problem.

The most standard approach for the Bernoulli problem is either variational or by singular perturbation. In the case $p=2$, the seminal paper by Alt-Caffarelli \cite{Alt-Caffarelli:minimum-problem} shows existence of a solution using variational techniques without any geometric restriction on the set $\partial\{u >0\}$. Moreover it is shown that the free boundary is flat, and consequently of class $\mathcal C^{1,\alpha}$ up to a hypersurface of $\HH^{n-1}$-measure zero. For $N=2$ they prove that the free boundary is globally analytic (see \cite{KN77}). Recently in \cite{Caffarelli-Jerison-Kenig-04} it was shown that there exists no variational solution with singular free boundary for $N=3$. The question remains open for $4\le N\le 6$, since  in \cite{DSJ09} existence of minimal cones is shown for $N=7$.

For $p\neq 2$ see \cite{DK10} and \cite{DP05}; in particular regularity is considered in  \cite{DP05,DP06}, where  the authors extend the results of \cite{Alt-Caffarelli:minimum-problem} ($\mathcal C^{1,\alpha}$ regularity near flat points). They prove that for $N=2$ the free boundary is globally analytic for $2-\delta < p <+\infty$ with $\delta$ an absolute constant. For $N\neq 2$ not many results are available: Petrosyan in \cite{P08} shows that for any $N\le k^*-1$ and any $2-\delta < p <2 +\delta$ the variational solution has globally analytic free boundary, where $k^*$  is the critical dimension above which any variational solution has free boundary with singular points.

If one adds a priori geometrical assumptions, the problem becomes more tractable. If the initial given domain is convex, for any $1<p<+\infty$ there exists a unique classical solution, and the free boundary is also convex and $\mathcal C^{2,\alpha}$ for both interior and exterior case. This is done in \cite{HH00,HH000,HH97} using Beurling's method on sub- and super-solutions.

Although the literature is not yet exhaustive, one would expect similar results if we relax the geometry of the domain to be star-shaped only (see, for instance \cite{Acker-Meyer:free-boundary-problem}). The $p$-capacitary potential in a star-shaped annulus has star-shaped level sets, according to the rearrangement results of \cite{Kawohl:starshaped}; see Section 5 for a more detailed literature on star-shaped Bernoulli problem. However, two questions remain open: \emph{(i)} Existence of a solution using Beurling's method, which, by uniqueness, would coincide with the variational one. \emph{(ii)} The best regularity one expects for the free boundary.

The paper is structured as follows: after recalling  in Section \ref{sec:preliminary} some preliminary results that will be used throughout the manuscript, we will consider the exterior and interior problem in the convex setting, see Sections \ref{section-exterior} and \ref{section-interior}, respectively. Our problem will be then considered for star-shaped domains in Section \ref{exterior-star}.
In Section \ref{sec:generalization} we further extend the problem to more general cases, such as non-constant distance constraint (including a constraint involving mean curvature), and two- and multi-phase problems. We also give the proof of a Brunn-Minkowski inequality. In the last section we show that the proposed model might give insights for new numerical algorithms that could more efficiently approximate solutions to the classical Bernoulli problem. Several open and tantalizing problems are also outlined.


\section{Preliminaries} \label{sec:preliminary}
\setcounter{equation}{0}
\setcounter{theo}{0}
For the readers' convenience we briefly recall in this section the main existence and regularity properties for $p$-capacitary functions that will be used later.

For $1<p<\infty$, we define the $p$-Laplacian operator $\Delta_p$ as
$$\Delta_p u:= \divergence (|\nabla u|^{p-2}\nabla u),\quad 1<p<\infty,$$
and the equation $\Delta_p u=0$ in $\Omega$ is understood in the weak sense as
$$\int_\Omega |\nabla u|^{p-2}\nabla u\nabla \varphi=0, \quad \forall\; \varphi\in W^{1,p}_0(\Omega).$$

We first review some classical regularity properties:

\begin{theo}\label{thm-classical-regularity}
Let $U$ be a bounded domain in $\mathbb R^N$ and $u$ be a (weak) solution of $\Delta_p u=0$ in $U$. Then:
\begin{enumerate}
\item[\emph{i.}]  $u\in\mathcal C^{1,\alpha}_{{loc}}(U),$ (\cite{DiBenedetto,Tolksdorf84}),
\item[\emph{ii.}]  $u$ satisfies the weak maximum principle and Hopf's boundary lemma (\cite{Tolksdorf83}).
\end{enumerate}
\end{theo}

Additional H\"older regularity up to the boundary can be shown through construction of conical barriers. With this objective in mind, we introduce some notations:
we say that a domain $U$ satisfies a \emph{uniform exterior cone condition} if there exists a $r_0>0$ such that for all $x^0\in\partial U$ the finite
right circular cone $V_{x^0}$ with vertex $x^0$ and opening $r_0$ satisfies $\overline U\cap V_{x^0}=\{x^0\}$ and $V_{x^0}\subset U^c$. One may analogously define the \emph{uniform interior
cone condition}.

\begin{lemma}[\cite{Tolksdorf83}]\label{lemma-holder-estimate}
Let $U$ be a domain with exterior cone condition. Then the solution of the Dirichlet problem with H\"older boundary data for the $p$-Laplacian in $U$ is a
H\"older-continuous function up to $\partial U$ with H\"older norm uniformly bounded depending on $r_0$, $\sup_U  |u|$ and $R_0$ (radius of largest ball that fits inside $U$).
\end{lemma}

With some additional information on the regularity of $\partial U$ one can
 obtain higher regularity for $u$ up to $\partial U$, in particular, uniform gradient bounds.

\begin{lemma}[\cite{Lieberman88}]\label{lemma-boundary-estimates}
Let $U$ be a $\mathcal C^{1,\alpha}$ domain in $\mathbb R^N$. Then any solution to $\Delta_p u=0$ in $U$ is $\mathcal C^{1,\beta}(\overline{U})$ where $\beta$ depends on $p$, $N$ and $\alpha$.
\end{lemma}

Finally, we say that a set $U$ satisfies the \emph{uniform interior ball condition} if there exists $r_0>0$ such that for each $x\in\partial U$, there exists $z_x\in U$
with the property that $x\in\partial B_{r_0}(z_x)\cap \partial U$ and $B_{r_0}(z_x) \subset U$.\\

In this paper we will be interested in the case that $u$ is a  $p$-capacitary potential in a convex annulus. Let $K,\Omega$ be two domains in $\mathbb R^N$, $N\geq 2$, satisfying $\overline K\subset \Omega$, $\Omega$ bounded.
 In the sequel we say that $u_\Omega$ is the {\em{p-capacitary}} potential of the ring shaped region $\Omega\setminus \overline K$ if:
$$\left\{
\begin{array}{lll}
\Delta_p\; \uO = 0  \;\hbox{ in }\; \Omega\setminus \overline{K},  && \\
\uO =1 \quad \textrm{on}\; \overline{K}, && \\
\uO =0 \quad \textrm{on}\; \partial\Omega. &&
\end{array}\right.\leqno ({\text{P}}_{\Omega})
$$

It is clear that a domain $\Omega\backslash\overline K$, with $\Omega$ and $K$ convex, automatically satisfies the uniform exterior cone condition, for some $r_0$ depending
only on $K$. Thus, the H\"older estimate from Lemma  \ref{lemma-holder-estimate} will only depend on the (interior) cone condition for the domain $K$ but not on $\Omega$.
We formulate this as a Corollary:

\begin{cor}\label{cor-holder-estimates}
Let $\Omega,K$ be two bounded convex domains such that $K\subset
\overline K\subset \Omega$, and denote by $u_\Omega$ the
$p$-capacitary potential of $\Omega\backslash \overline K$ as given
in equation $(P_{\Omega})$.  Suppose that $K$ satisfies the uniform
interior cone condition for opening $r_0$, and set
$d_0:=\dd(\partial\Omega,K)$.
 Then there is a constant $M=M(r_0,d_0)$ such that
$$\|u_\Omega\|_{\mathcal C^{\alpha}(\overline{\Omega\backslash K})}\leq M.$$
\end{cor}

If, in addition, $K$ is $\mathcal C^{1,1}$, one can obtain uniform gradient estimates up to the boundary:
\begin{cor}\label{lemma-gradient-above}
Let $\Omega,K$ be two bounded convex domains such that $ \overline K\subset \Omega$, and denote by $u_\Omega$ the $p$-capacitary potential of $\Omega\backslash K$ as given in $(P_{\Omega})$.  Suppose that $K$ satisfies the uniform interior ball condition for radius $r_0$. Denote also $d_0=\min_{x\in\partial\Omega} \dd(x,K)$. Then there is a constant $M=M(r_0,d_0)$ such that
$$|\nabla u_\Omega|\leq M\quad in \quad  \Omega\backslash\overline K.$$
\end{cor}

Now we deal with the lower bounds for the gradient:

\begin{lemma}[Lemma 2 in \cite{Lewis77}]\label{lem-gradient-below}
Let $\Omega,K$ be two bounded convex domains such that $ \overline K\subset \Omega$ and denote by $u_\Omega$ the $p$-capacitary potential of
 $\Omega\backslash \overline K$ as given in $(P_{\Omega})$.  Suppose that $\Omega\backslash \overline K$ satisfies the uniform interior ball condition for radius $r_0$.
  There is a constant $M_0$ such that
$$|\nabla u_\Omega|\ge M_0\quad in \quad  \Omega\backslash\overline K.$$
\end{lemma}

If $K$ is uniformly convex, it can be easily shown by constructing a suitable barrier that the gradient of $|\nabla u_\Omega|$ near $\partial K$ must be bounded from
below by a positive constant. If we do not have this extra information on $K$, but we know the behavior near $\Omega$, we can translate this information back to $K$.  First we need to introduce some notation. For two nested convex sets $K\subset \Omega$, and for $x\in\partial K$, we denote by $T_{x,a}$ the supporting hyperplane at $x$ with the normal $a$ pointing away from $K$. Depending on the geometry of $\partial K$, $T_{x,a}$ is not necessarily unique. Now for each $x\in \partial K$ there corresponds a point $y_x$, not necessarily unique, on $\partial \Omega\cap \{z\,:\, a\cdot(z-x)>0\}$ such that $a\cdot(y_x-x)=\max a\cdot (z-x)$, where the maximum has been taken over all $z\in \partial\Omega \cap \{z\,:\, a\cdot(z-x)>0\}$. Then, a crucial result is:

\begin{lemma}[Lemma 2.2. in  \cite{HH000}]\label{lemma-gradient-below}
Let $K$ and $\Omega$ be two nested open convex domains, $\overline K\subset \Omega$, and denote by $u_\Omega$ the $p$-capacitary potential of $\Omega\backslash K$, as given in problem $(P_{\Omega})$. Then
$$\limsup_{\stackrel{z\to x}{z\in\Omega\backslash\overline K}} |\nabla u(z)|\geq \limsup_{\stackrel{z\to y_x}{z\in\Omega\backslash\overline K}}  |\nabla u(z)| \quad  \mbox{for all }x\in\partial K,$$
where $y_x$ is the point indicated in the previous discussion.
\end{lemma}




Next, a fundamental result on properties of level sets for $p$-capacitary functions $u_\Omega$ in a convex ring:

\begin{prop}[\cite{Lewis77}]\label{convexity-level-sets}
Let $l\in [0,1)$. Given a convex ring $\Omega\backslash \overline K$, its $p$-capacitary function $u_\Omega$ described as above is real analytic and the level sets $\{u_\Omega>l\}$ are convex (with strictly positive principal curvatures) and analytic.
\end{prop}

Finally, a crucial technical result that allows to define the gradient of the solution at the boundary of a convex ring, and thus, give a notion of classical solutions for the usual Bernoulli problem:

\begin{prop}[Theorem 1.3 in \cite{HH02}]
Let $K$ and $\Omega$ be two nested open convex domains, $ \overline K\subset \Omega$, and denote by $u_\Omega$ the $p$-capacitary potential of $\Omega\backslash \overline K$, as given in problem $(P_{\Omega})$. Suppose also that $\partial K$ and $\partial \Omega$ are $\mathcal C^1$. Then $\nabla u$ us semi-continuous in $\overline{\Omega\backslash K}$, and non-tangentially continuous up to $\partial\Omega\cup\partial K$.
\end{prop}

Finally, we recall the definition of extremal points:
\begin{defin}\label{extr-points}
Let $\Omega$ be a convex set: a point $x\in
\partial \Omega$ is an {\em{extremal}} point if $x$ cannot be
written as a linear combination of the form $x=tx_1+(1-t)x_2$ for
$t\in(0,1)$, $x_1,x_2\in\partial\Omega$.
\end{defin}

In the rest of the paper we denote by $ \Gamma_l$ the $l$ level set of $u$:
$$ \Gamma_l:=  \{x\in \Omega \; | \; u(x)=l\}.$$
 We also write $K\Subset \Omega$ when $\overline K\subset \Omega$.


\section{The exterior problem in convex setting}\label{section-exterior}

\setcounter{equation}{0}
\setcounter{theo}{0}

The first aim of this paper is to solve the following problem: given a convex open bounded set $K\subset \R^N$, $N\ge 2$, and $l\in(0,1)$, $\lambda>0$ positive constants, find a function $u$ and a convex open bounded domain $\Omega\subset\mathbb R^N$, $\Omega \supset \overline K$, solving
\begin{equation*}\left\{\begin{split}
& \Delta_p u = 0 \hbox{ in } \Omega\setminus \overline{K},  \quad
u =1  \textrm{ in } \overline{K}, \quad u =0 \textrm{ on } \partial\Omega, \\
&\dd( x,\Gamma_l) =\lambda  \quad\textrm{for all}\; x\in\partial \Omega.
\end{split}\right.\leqno (P_E)
\end{equation*}

We first show:

\begin{theo}\label{thm-exterior}
Problem $(P_E)$ has a unique regular solution $(u,\Omega)$ with $\Omega$ convex, $\partial \Omega \in \mathcal C^\infty$, and $u\in \mathcal C^{1,\alpha}(\overline{\Omega\backslash K})$ uniformly up to the boundary of $\Omega$.
\end{theo}

The proof is inspired by previous works of one of the authors \cite{HH00,HH000} via Beurling's method (see also the classical reference \cite{Beurling89}). The method is based on constructions of super- and sub-solutions for $(P_E)$; hence, existence of a solution for $(P_E)$ follows by the proof of the existence of a {\em minimal} super-solution which turns out to be a subsolution (or solution) as well.


\subsection{Proof of existence} \label{sec:existence}

We first deal with the existence statement in Theorem \ref{thm-exterior}. For each convex set $\Omega$, let $u_\Omega$ be the $p$-capacitary potential of $\Omega\backslash \overline K$. Let us consider the class
$$\CC:= \{ \Omega \;{\textrm{convex bounded open set in }}\mathbb R^N, \; \Omega \supset \overline K \},$$
and define the following three subsets:
\begin{equation*}\begin{array}
{ll}
\AA &= \{ \Omega \in\CC \; | \;   \sup\limits_{x\in\partial\Omega}\dd( x,\Gamma_l) \le \lambda \},\\
\AA_0 &= \{ \Omega \in \CC \; | \;   \inf\limits_{x\in\partial\Omega}\dd( x,\Gamma_l) <\lambda \}, \\
\BB &= \{ \Omega \in \CC \; | \;   \inf\limits_{x\in\partial\Omega}\dd( x,\Gamma_l) \ge \lambda \}.
\end{array}
\end{equation*}
Note that $\AA$ is the set of subsolutions, $\AA_0$ the set of strict subsolutions, while  $\BB$ is the set of supersolutions for problem $(P_E)$. Let us first check that they are non-empty sets :

\begin{lemma}\label{lemma-sub-supersolution}
There exist a supersolution and a strict subsolution for problem $(P_E)$.
 \end{lemma}

\begin{proof}
By translation suppose that the origin of the cartesian system lies inside $K$. In order to construct a supersolution take  $0<r<R$ big enough such that $K\subset B_r(0)$. The solution of the capacitor problem for the annulus $ B_R(0)\backslash \overline{B_r(0)}$ is explicitly given by function
\begin{equation*}\left\{
\begin{split}
&\overline u(|x|) = \frac{ \log\lp\frac{R}{|x|}\rp}{   \log\lp\frac{R}{r} \rp} \quad\textrm{for}\;p=N, \\
&\overline u(|x|) = \frac{|x|^{\frac{p-N}{p-1}}-R^{\frac{p-N}{p-1}}}{r^{\frac{p-N}{p-1}}-R^{\frac{p-N}{p-1}}}\quad \textrm{for}\;p\neq N.
\end{split}\right.
\end{equation*}
It is clear that we can have $\dd(\{\bar u=l\},\{\bar u=0\})$ as big as we wish by fixing $r$ and taking $R$ big enough. Consider also $u_R$ to be the
 $p$-capacitary potential of the set $ B_R(0)\backslash \overline K$. By the comparison principle, $u_R\leq \bar u$ for all $x\in B_R(0)\backslash B_r(0)$.
 This immediately implies that
 \begin{align}\label{formula_used}
 \dd(x,\{u_R=l\})\geq \lambda\quad \mbox{for all }x\in\partial B_R(0),
 \end{align}
 which shows that $u_R$ is a supersolution for our problem, i.e, $B_R(0)\in\BB$.

Now we seek a strict subsolution. Let $B_R(0)$ and $u_R$ as above. Since $K\subset B_R(0)$, and $u_R$ is the $p$-capacitary potential in this annulus, we can apply Lemma \ref{lem-gradient-below} to bound the gradient  of $u_R$ from below. More precisely, there exists a small neighborhood $U$ of $\partial K$ and some $\alpha>0$ such that
\begin{equation}\label{gradient-below}|\nabla u_R(x)|\geq \alpha>0\quad\mbox{for all}\quad x\in U\backslash K.\end{equation}
Choose the level set $\{u_R=1-\eps\}$ and set $U_\eps:=\{x:u_R(x)>1-\eps\}\cap U$.  Define now
$$u_\eps(x) := \frac{ u_R(x) - (1-\eps)}{\eps}.$$
The function $u_\eps(x)$ is a $p$-harmonic function in the domain $U_\eps$, $u_\eps(x) =0$ in  $\{ u_R(x) = 1-\eps \}$ and $u_\eps(x)=1$ in $K$, so it is the $p$-capacitary potential of the set $U_\eps$. We need to prove that $u_\varepsilon \in\AA_0$, i.e.,
$$\sup_{x\in \{ u_\eps =0\} }\dd ( x,\{ u_\eps =l\}) < \lambda .$$
This is easy to see since from the lower bound \eqref{gradient-below} it follows
$$
\dd (x, \{ u_\eps =l\}) \leq \frac{\eps l}{\alpha} < \lambda, \quad x\in\{ u_\eps =0\},$$
by choosing $\eps$ small enough.
\end{proof}

\begin{lemma}\label{lemma-sub-super}
Let $(u_S,\Omega_S)$ be any supersolution and $(u_s,\Omega_s)$ any subsolution for problem $(P_E)$. Then $\Omega_s\subset \Omega_S$ and $u_s\leq u_S$.
\end{lemma}

\begin{proof}
We use Lavrent'ev rescaling method. Assume by contradiction that there is a subsolution $(u_s,\Omega_s)$ which is not smaller than a supersolution $(u_S,\Omega_S)$. Without loss of generality, assume that $0\in K$. Rescale the function $u_s$ by a parameter $\eps<1$ small enough, i.e. $u_s^\eps (x)= u_s(\frac{x}{\eps})$, so that   $u_s^\eps$ is a $p$-capacitary function in $\Omega_s^\eps\backslash \overline{K_\eps}$ for $\Omega_s^\eps \subset \Omega_S$, $K^{\eps}\subset K$. Moreover, for all $x\in\partial \Omega_{s}^\eps$,
\begin{align*}
\dd(x, \{u_s^\eps=l\}) \le \lambda^\eps <\lambda.
\end{align*}

Consider the biggest $\eps_0$ such that $\Omega_s^{\eps_0} \subseteq \Omega_S$ and $\partial\Omega_s^{\eps_0} \cap  \partial\Omega_S\neq 0$. Note that we can take $\eps_0<1$ because $\Omega_s\neq\Omega_S$. Then, by strong comparison principle $u_s^\eps\le u_S$ for each $x\in \Omega_s^{\eps_0}\backslash K$. Let $x^0\in \partial\Omega_s^{\eps_0} \cap \partial\Omega_S$: it holds that
 $$\lambda >\lambda^\eps  \ge \dd(x^0 ,\; \{u_s^\eps=l\}) \ge  \dd(x^0,\;\{u_S=l\}) \ge \lambda,
 $$
which is a contradiction.
\end{proof}

The next step in the application of Beurling's method is to show that the set  $\BB$ is closed under intersection. This is the preparation step in order to show that the limit of a decreasing sequence of elements in $\BB$ still belongs to $\BB$.

\begin{lemma} \label{lemma-intersection}
Let $\Omega_1$ and $\Omega_2$ be two elements of $\BB$. It holds that
$$
\Omega_1 \cap \Omega_2 \in \BB.
$$
\end{lemma}

\begin{proof}
Let $u_{\Omega_1 \cap \Omega_2 }$ be the solution of the $p$-capacitary problem in the intersection $\Omega_1 \cap \Omega_2$.  By application of the maximum principle in each of the domains $\Omega_1$ and $\Omega_2$, we can conclude that
$$u_{\Omega_1 \cap \Omega_2 }(x)  \le \min\{u_{\Omega_1 }(x),u_{\Omega_1 }(x)\}, \quad \textrm{for each} \; x\in \Omega_1 \cap \Omega_2.$$
Moreover, since  $\Omega_1\cap\Omega_2$ must contain  $\overline K$, the level set $\{u_{\Omega_1\cap\Omega_2}=l\}$ is well defined.
Let $x^0\in \{u_{\Omega_1\cap\Omega_2}=l\}$ the point where the minimum distance between  the sets $\partial(\Omega_1\cap\Omega_2)$ and $\{u_{\Omega_1\cap\Omega_2}=l\}$ is achieved. Then
$$l=u_{\Omega_1\cap\Omega_2}(x^0)\leq \min\{u_{\Omega_1}(x^0),u_{\Omega_2}(x^0)\},$$
which implies that the point $x^0$ is further away from both boundaries $\partial\Omega_1$ and $\partial\Omega_2$, in other words  a distance greater than or equal to $\lambda$.
\end{proof}

Next lemma shows convergence of the sets $\Omega_n$.

\begin{lemma}\label{lemma-convergence-exterior}
Let $\Omega_1\supset\Omega_2\supset\ldots$
be a decreasing sequence of convex domains in $\BB$, and suppose that the set $\Omega$, defined the as interior of the closure
$\overline{\cap\Omega_k}$, is convex, bounded, open and $\Omega\supset \overline K$. Then $\Omega\in\BB$.
\end{lemma}

\begin{proof}
Let $u_k$ be the $p$-capacitary potential of $\Omega_k\backslash \overline K$. By construction, $0\leq u_k\leq 1$, and it is a decreasing sequence.
Then we can easily show that on every compact subset of $\Omega:=\mbox{Interior}(\overline{\cap \Omega_k})$, $u_k$ converges in $\mathcal C^{1,\alpha}$ norm to a $p$-harmonic function $u$.
However, the lack of control of the gradients at the boundary does not allow to show higher regularity up to the boundary, but only $\mathcal C^\alpha$.

 We first prove that $u$ is precisely the solution to the $p$-capacitary problem in $\Omega\backslash \overline K$, denoted by $u_\Omega$. Clearly, $u\equiv 1$ in $\overline K$.
  We also need to check that $u\equiv 0$ outside $\Omega$. For that, let $x\not\in\Omega$.
Since $K$ is convex, $\Omega_k\backslash\overline K$ automatically satisfies the uniform exterior cone condition, and one may use Corollary \ref{cor-holder-estimates}
to estimate
\begin{align}\label{u_Null}
0 &\leq u_k(x)\leq \dd(x,\R^N \setminus \Omega_k)^\alpha\sup \|u_k\|_{\mathcal C^{\alpha}(\overline \Omega\setminus K)}\\
 &\leq M \dd(x,\R^N \setminus \Omega_k)^\alpha\to 0,\nonumber
\end{align}
as $k\to\infty$. This implies that $u(x)=0$, as desired.

Since the sequence of functions is decreasing, the set $\{u_k=l\}$ converges to the set where $\{u=l\}$. However the boundary of the set $\Omega_k$ may not converge to the set $\{u=0\}$: that may happen if the gradient of $u$ vanishes at some point of $\partial\Omega$ so that the limit function flattens out, and in this case the distance property would be violated. As a consequence, we need to check that $u(x)>0$ for all $x\in\Omega$.
By the minimum principle, $u$ cannot vanish in the interior, but it could happen that it vanishes at the boundary of some set $\tilde\Omega\subset\Omega$.  Hence assume by contradiction that $\tilde\Omega\neq\Omega$; consider a point $x^0\in\partial\tilde\Omega$, $x^0\not\in\Omega$ and  let $P\in\{u=l\}$ be the point where distance between $x^0$ and $\{u=l\}$ achieves its minimum. It can happen that $x^0=P$ but this is not a problem in the argument.
Note that the ball $B_{\eps/2}(x^0)$ is completely inside $\Omega_k$ for all $k$ (otherwise, the distance property for $u_k$ would be violated). Moreover $u_k(x)\to 0$ as $k\to\infty$ for all $x\in B_{\eps/2}(x^0)\setminus(B_{\eps/2}(x^0) \cap \tilde\Omega) $.
Cover the segment $\overline{x^0 P}$ by a finite number of balls of radius $\eps/2$, and define $S=\cup_{i=0}^N B_{\eps/2}(x^i)$. In the set $S$ the Harnack inequality applied to the function $u_k$ implies
$$l=u(P)\leq u_k(P)\leq \sup_S u_k(x)\leq C \inf_S u_k(x),$$
where $C$ depends only on $S$. But this is a contradiction since, as $k\to\infty$, it holds
$$l\leq C\inf_S u_k(x)\to 0.$$
We have shown that $\tilde\Omega=\Omega$,  so that $u>0$ on
$\Omega$. Then $u=u_\Omega$ is the $p$-capacitary potential for the
set $\Omega\backslash \overline K$. Moreover the set $\Omega$
belongs to $\BB$ since
\begin{align}\label{conv-level-set}
\lambda \le \dd(\partial \Omega_n , \{u_n=l\}) \to \dd(\partial \Omega , \{u=l\}) \quad \textrm{as}\; n\to +\infty.
\end{align}
Note here that the passage to the limit is justified because  $\partial\Omega_n\to\partial\Omega$ and  $\{u_n=l\}\to \{u=l\}$ in the Hausdorff distance sense.
\end{proof}

The previous two lemmas allow us to find a minimal set $\Omega$ in $\mathcal B$, i.e., the solution $u$ will be the minimal supersolution,
 that stays above the strict subsolution we have found in Lemma \ref{lemma-sub-supersolution}. More precisely we have the following corollary:

\begin{cor} \label{cor-exists-minimal}
Assume that there exist two domains $\Omega^0\in\AA_0$ and $\Omega^1\in\BB$, with $\Omega^0\subset\Omega^1$. Set
$$\SS:=\{\Omega\in\BB \,:\, \Omega^0\subset \Omega\}.$$
Then there exists a domain $\Omega$ in the class $\SS$ which is minimal for the inclusion.
\end{cor}

\begin{proof}
Let $\tilde\Omega$ be the intersection of all domains in the class $\SS$ and set $\Omega$ to be the interior of the closure of $\tilde\Omega$, which is still convex. To prove $\Omega\in\BB$, we select a sequence of domains $\{U_k\}_{k=1}^\infty$ in $\SS$ such that $\cap U_k=\tilde\Omega$. Consider the sequence of domains defined by $\Omega_1=U_1$ and $\Omega_{k+1}=\Omega_k \cap U_{k+1}$ for $k\geq 1$. Each $\Omega_k$ is convex and because of Lemma \ref{lemma-intersection} belongs to $\BB$. Lemma \ref{lemma-convergence-exterior} completes the proof of the corollary.
\end{proof}

In the following, we always denote by $\Omega$  the (non-trivial)
minimal element in the class $\SS$ defined in the previous
proposition. To finalize the proof of the theorem  we need to check
that this minimal set satisfies the distance property so it is
indeed the solution of $(P_E)$.

\vspace{3mm} \noindent {\bf Proof of Theorem 3.1:}\newline We need
to show that $\Omega$,  the minimal set of $\BB$, and its
corresponding $p$-capacitor potential  $\uO$ satisfy the distance
property:
$$\dd( x,\{ {\uO}=l \}) =\lambda\quad\mbox{for all}\quad x\in\partial\Omega.$$
The proof for the standard Bernoulli problem uses crucially the convexity assumption. Our problem may be solved through a different and much simpler argument, which may be extended to more complex geometries, for instance, the star-shaped case. By contradiction, suppose that there is a point $x^0\in \partial \Omega$ such that
$$
\dd( \{ {\uO}=l \}, \; x^0)  = \lambda+\eps,\quad \eps>0.
$$
Set $\mathcal L:=\{\uO>l\}$. Note that $\mathcal L\backslash \overline K$ is a convex ring. Define $\tilde\Omega$ as the set
$$\tilde\Omega:=\{x\in\Omega : \dd(x,\mathcal L)<\lambda\},$$
with $\tilde u$ the $p$-capacitary potential of $\tilde\Omega\backslash\overline K$. By the hypothesis, $\tilde \Omega\subsetneqq \Omega$, and hence the comparison principle gives that $\tilde u\leq u$. One may estimate, for  $x\in\partial\tilde\Omega$, that
$$\dd(x,\{\tilde u=l\})\geq \dd(x,\{u=l\})\geq\lambda.$$
Since $\partial \tilde\Omega$ lies at fixed distance $\lambda$ from
$\mathcal L$, and $\mathcal L$ is a convex bounded set, we can
conclude that  also $\tilde\Omega$ is convex. Hence $\tilde\Omega$
belongs to the set $\BB$; but this contradicts the fact that
$\Omega$ is the minimal set in $\BB$.

\begin{lemma}\label{lemma_uniq}
Assume that $K$ is star-shaped. If  problem $(P_E)$ admits a solution, then it is unique.
\end{lemma}

\begin{proof}
It is a consequence of Lemma \ref{lemma-sub-super}, since a solution is at the same time a sub- and a super-solution.
\end{proof}

\begin{rem}
The uniqueness result together with the constructive construction of $\Omega$ as intersection of convex sets give that the solution $\Omega$ of problem $(P_E)$ must be convex.
\end{rem}

We look now at regularity. It is clear that $u\in\mathcal C^{\infty}(\Omega\backslash\overline K)$ and
even analytic in the interior of $\Omega\backslash K$, but what about the regularity of the free boundary?
First note that $\{u=l\}$ is level set of an analytic function, hence it is analytic. Moreover all the normal curvatures at the level set are positive (Lemma \ref{convexity-level-sets}). The level set  $\{u=0\}$ is at a fixed distance $\lambda$ from  $\{u=l\}$ and its curvatures are of the form $\frac{\mu_i}{1 + \lambda \mu_i}$, where $\mu_i$ are the curvatures of $\{u=l\}$; hence we can conclude \cite{Foote:distance-function} that $\{u=0\}$ is also $\mathcal C^\infty$ and analytic.\\

The proof of Theorem \ref{thm-exterior} is thus completed.

\subsection{Convergence to the Bernoulli problem}\label{sec:uniqueness}

Finally, we show that the discrete Bernoulli problem converges, at the limit, to the usual Bernoulli problem:

\begin{theo}
Let $l_n$ and $\lambda_n$ be two decreasing sequences of positive numbers such that $l_n\to 0$, $\lambda_n\to0$ as $n\to +\infty$ and $ l_n = \lambda_n\omega$.  Let $(u_n,\Omega_n)$ be the solution of the exterior problem $(P_E)$ with constants $l =l_n $ and $ \lambda=\lambda_n$. As $n\to \infty$, $\{\Omega_n\}$ is a decreasing sequence of sets converging in Hausdordff distance to some convex set $\Omega$ and the function $u_n$ converges in $\mathcal C^{1,\alpha}_{\mbox{loc}}(\Omega\backslash \overline K)\cap\mathcal C^{\alpha}(\overline{\Omega\backslash K})$ to the (unique) solution $u$ of the classical Bernoulli-type problem
\begin{equation*}
\left\{\begin{split}
&\Delta_p u = 0 \hbox{ in } \Omega\setminus \overline{K},  \quad u =1 \textrm{ in } \overline{K}, \quad u =0 \textrm{ on }\partial\Omega, \\
&|\nabla u| =\omega,  \quad\textrm{for all }\; x\in\partial \Omega.
\end{split}\right.\leqno ({P_B})
\end{equation*}
The boundary condition $|\nabla u| =\omega$ is to be understood in the following sense:
\begin{equation}\label{Bernoulli-condition}
\liminf_{y\to x, y\in\Omega} |\nabla u(y)|=\limsup_{y\to x,y\in\Omega} |\nabla u(u)|=\omega, \quad  \mbox{for every }x\in\partial\Omega.
\end{equation}
\end{theo}

\begin{rem} It is shown in \cite{HH000} that the solution for the Bernoulli problem $({P_B})$ exists, it is unique if $K$ is bounded, and has $\mathcal C^{2,\alpha}$ boundary $\partial\Omega$.
\end{rem}

\begin{proof}
Define the rescaled distance function
$$d_n(x) = l_n[1 - \dd(x, \{u_n=l_n\})/\lambda_n].$$
Clearly, when $x \in \partial \Omega_n \cup\{u_{n}=l_n\}$ we have that $d_n(x) = u_n(x).$ Define also $S_n : = \{ 0\le u_n \le l_n\}$, a narrow ring. Since the level sets of $u_n$ are convex surfaces, the function $d_n$ is $p$-superharmonic in the set $S_n$, i.e,
$$-\Delta_p d_n \geq 0.$$
Since $d_n$ and $u_n$ have the same boundary values on $\partial S_n$, comparison principle yields that $u_{n}\leq d_n$ everywhere in $S_n$.
 In particular, one obtains that for $x\in\partial\Omega_{n},$ and $l_m\leq l_n$,
$$\dd(x,\{u_{n}=l_m\})\geq \dd(x,\{d_n=l_m\}).$$
On the other hand for any $x\in \;\{0\leq d_n \leq l_n\}$ it holds $|\nabla d_n(x)|=\omega$. Since $0\leq l_m\leq l_n$,
$$\dd(x,\{d_n=l_m\})\geq \frac{l_m}{\sup|\nabla d_n|}=\frac{l_m}{\omega}=\lambda_m\quad \mbox{for all }\quad m\geq n.$$
Thus, $u_{n}$ is a supersolution for every problem $(P_E)$ with constants $(\lambda_m,l_m)$, $m\geq n$, i.e., $\Omega_{n}\in\BB_m$ where
\begin{equation*}
\BB_m:= \{ \Omega\in\CC \;|\;   \inf\limits_{x\in\partial\Omega}\dd( x,\{\uO=l_m\}) \ge \lambda_m \},
\end{equation*}
for all $m\geq n$. Because of our construction of the solution for the exterior problem with constants $(\lambda_m,l_m)$ as the minimal set in $\BB_m$,
this immediately gives that
$\Omega_{n}$ contains all the $\Omega_m$ for $m\geq n$. Consequently we have shown that the sequence of $\{\Omega_n\}_{n=0}^{\infty}$ is decreasing and from maximum principle it follows that $\{u_n\}$ is a decreasing sequence.

Since $u_{n}\leq d_n$ on $S_n$, using Hopf's boundary lemma we arrive at  $|\nabla u_{n}| \leq |\nabla d_n| = \omega$ on $\partial\Omega_n$. This tells us that $u_n$ is a supersolution for the Bernoulli problem with constant $\omega$ in the sense indicated in \cite{HH000}.

As a consequence, the sequence of sets $\Omega_n$ converges (for example in the Hausdorff measure) as $n\to \infty$ towards a convex domain $\Omega\supset \overline K$. Moreover $u_n$ converges in $\mathcal C^{1,\alpha}_{loc}$-norm to a $p$-harmonic function $u$: by Theorem 3.2 in \cite{HH000} we know that $u$ solves the $p$-capacitary problem in $\Omega \setminus K$ and is a supersolution, i.e., $|\nabla u|\leq \omega$ for all $x$ in $\partial\Omega$.

The last step consists of showing that $(u,\Omega)$ solves the Bernoulli problem $(P_B)$ with constant $\omega$ in the sense indicated in \eqref{Bernoulli-condition}. Use again comparison to the distance function, we obtain that $|\nabla u_n| \geq \omega$  on the level set $\{u_n=l_n\}$. But Lemma \ref{lemma-gradient-below} gives that the inequality $|\nabla u_n| \geq \omega$ is also true for every point in  $\{l_n \leq u < 1\}$. Since we have uniform convergence of the gradients away from the boundary and the sequence of sets $\Omega_n\to \Omega$ is decreasing, we are able to pass to the limit and thus $|\nabla u(y)|\geq \omega$ for every point $y\in\Omega\backslash \overline K$. The proof of the Theorem is completed.
\end{proof}

\section{The interior problem in convex setting}\label{section-interior}
\setcounter{equation}{0}
\setcounter{theo}{0}

We consider now the following problem: given a convex open bounded set $\Omega\subset \R^N$ for $N\ge 2$ and  $l\in(0,1)$, $\lambda>0$, find a function $u$ and a convex open bounded domain $K \subset \Omega$ such that
\begin{equation*}\left\{
\begin{split}
&\Delta_p u = 0  \hbox{ in } \Omega\setminus \overline{K},  \quad u =0 \textrm{ in } \overline{K}, \quad u=1 \textrm{ in } \partial\Omega, \\
&\dd( x,\Gamma_l) =\lambda,  \quad \textrm{for all}\; x\in\partial K.
\end{split}\right.\leqno ({P_I})
\end{equation*}

Our first main result for the interior case is the following.

\begin{theo}\label{thm-interior}
There exists a constant $\lambda_{\Omega,max}$ that depends only on $l$ and $\Omega$  such that for any $\lambda \le \lambda_{\Omega,max}$ problem $(P_I)$ has a smooth solution $(u,K)$ with $K$ a convex set and $\partial K \in \mathcal C^{0,1}$.
\end{theo}

For each convex set $K$, one can define  its associated $p$-capacitary potential $u_K$ as the solution of the following problem:
\begin{equation*}
\left\{
\begin{split}
&\Delta_p u_K = 0  \hbox{ in } \Omega\setminus \overline{K},  \\
&u_K =0 \textrm{ in }\overline{K},  \\
&u_K =1  \textrm{ on } \partial\Omega.
\end{split}\right.\leqno ({\text{P}}_K)
\end{equation*}
Note that,  without risk of confusion, we have switched the notation from problem $(P_{\Omega})$ in the exterior case.

The proof of Theorem \ref{thm-interior} goes, as in the exterior case, by considering the maximal element in the class:
\begin{equation}\label{supersolutions}
\begin{split}\BB_\lambda = \{& K \;{\textrm{convex bounded open set}},\\
 &\; \Omega \supset K \; | \;   \inf_{x\in\partial K}\dd( x,\{u_K=l\}) \ge \lambda \}.
\end{split}
\end{equation}
In the following  we show that, fixed $l\in (0,1)$, the set $\BB$ is non-empty if the value of $\lambda$ is smaller than a certain critical value, which will be denoted by $\lambda_{\Omega,max}$. We also show that problem $(P_I)$ does not necessarily have a unique solution for certain values of $\lambda$.

\subsection{A minimal supersolution and the Bernoulli constant}

 As a motivation, consider the simple case that $\Omega$ is a ball: such particular example shows, as in the case of the classical Bernoulli problem, the existence of a constant $\lambda_{\Omega,max}$ such that no solution exists for $\lambda > \lambda_{\Omega,max}$, one or two solutions can occur for $\lambda<  \lambda_{\Omega,max}$ and only one solution exists if $\lambda =\lambda_{\Omega,max}$. Using Beurling's terminology, we call the unique solution that corresponds to the value $\lambda = \lambda_{\Omega,max}$ of {\em{parabolic}} type, the smallest solution for $\lambda <\lambda_{\Omega,max}$ of {\em{hyperbolic}} type and the biggest solution for $\lambda <\lambda_{\Omega,max}$ of {\em{elliptic}} type. The elliptic solutions form a decreasing family, while the hyperbolic solutions form a increasing family of solutions.

\begin{lemma} \label{lemma_ball}
Let $\Omega = B_R(0)\subset \mathbb R^N$ be a ball with radius $R$. Fixed $l\in(0,1)$ and $p>1$, we have that:
\begin{itemize}
\item[\emph{i.}] If $p\le N$, there exists a constant $\lambda_{max}$ depending on $l$ and $R$ such that problem $(P_I)$ has a unique solution for $\lambda = \lambda_{max}$, two solutions if $0<\lambda< \lambda_{max}$ and no solutions if $\lambda>\lambda_{max}$.
\item[\emph{ii.}] If $p> N$, then there exist constants $\lambda_{max}$ and $\lambda_{min}$ depending on $l$, $p$ and $R$ such that problem $(P_I)$ has a unique solution for $\lambda\in(0,\lambda_{min})\cup \{\lambda_{max}\}$, two solutions if $\lambda\in(\lambda_{min},\lambda_{max})$ and no solutions if $\lambda>\lambda_{max}$.
\end{itemize}
\end{lemma}

\begin{proof}
It has been shown in \cite{Reichel95} that all the solutions of the $p$-capacitary problem with $\Omega = B_R(0)$ are balls centered in the origin. We deal first with the case $N=p$. Given any $r$ with $0<r<R$, the unique solution $u_r$ of the $p$-capacitary problem with $\Omega = B_R(0)$ and $ K = B_r(0)$ reads as follows:
$$
u_r(|x|) = 1 - \frac{ \log(R)- \log(|x|)}{\log(R) - \log(r)}.
$$
It holds
$$
u_r(|x|) = l \quad \textrm{if and only if} \quad |x| = R^l\; r^{1-l},
$$
and consequently
$$
\Lambda(r) := \; \dd( x,\{u_r=l\}) = r \; \left( \left( \tfrac{R}{r}\right)^l -1\right)  \quad\textrm{for all}\; x\in\partial B_r(0).
$$
We note here that $\Lambda(r)+r\leq R$ for $r\in(0,R)$ so $u_r$ is an admissible solution for the $p$-capacitary problem in the annulus.

With elementary computations one can check that $\Lambda(r)$ reaches its maximum value at
\begin{align}\label{Dmax}
r_{max} := R \; ( 1-l) ^{\frac{1}{l}} \quad \textrm{with} \quad \Lambda(r_{max} ) = R\frac{ ( 1-l) ^{1/l}}{ \left(1/l -1\right)}=:\lambda_{max}.
\end{align}
We may easily compute $\Lambda(0)=0$, $\Lambda(R)=0$, and $\Lambda(r)$ is strictly increasing for $0< r <r_{max}$, strictly decreasing for $ r_{max} <r<R$. This implies that for any $\lambda \in (0,\; \lambda_{max} ) $ there exist two values $r_{\lambda,1} < r_{max} < r_{\lambda,2}<R $ such that both solutions $u_i$, $i=1,2$ of the $p$-capacitary problems in $B_R(0)\setminus \overline{B_{r_{\lambda,i}}(0)}$
satisfy $\dd( x,\{u_i=l\}) =\lambda  \quad\textrm{for all}\; x\in\partial B_{r_i}(0)$ for $i=1,2$. Moreover, there exists a unique solution for $\lambda_{max}$, and no solutions in the remaining case. This completes the proof of the lemma for the case $p=N$.\\

On the other hand, when $p\neq N$ the $p$-capacitary potential is given by
$$u_r(|x|)=\frac{r^\frac{p-N}{p-1}-|x|^{\frac{p-N}{p-1}}}{r^\frac{p-N}{p-1}-R^{\frac{p-N}{p-1}}},$$
and, keeping the same notation as above,
\begin{align*}\Lambda(r)=\left[ (1-l)r^{\frac{p-N}{p-1}}+l R^\frac{p-N}{p-1}\right]^{\frac{p-1}{p-N}}-r.
\end{align*}
But again, $\Lambda(r)+r\leq R$, so $u_r$ is indeed an admissible solution in the annulus.

Next, $\Lambda(r)$ reaches its maximum value at
\begin{align*}
r_{max} := {R}\left[ \frac{l}{(1-l)^{\frac{N-p}{N-1}}-(1-l)}\right]^{\frac{p-1}{p-N}}.
\end{align*}
Note that $\Lambda$ is an increasing function in the interval $(0,r_{max})$, decreasing at $(r_{max},R)$, $\Lambda(R)=0$ and $\Lambda(0)=l^{\frac{p-1}{p-N}}R>0$ if $p>N$ and $\Lambda(0)=0$ for $p<N$.

We set  $\lambda_{min}:=\Lambda(0)$ and
\begin{equation}\label{lambda-max2}\lambda_{max}:=\Lambda(r_{max})
=R
\left[\left(\tfrac{l}{1-(1-l)^{\frac{p-1}{N-1}}}\right)^{\frac{p-1}{p-N}}
-\left(\tfrac{l}{(1-l)^{\frac{N-p}{N-1}}-(1-l)}\right)^{\frac{p-1}{p-N}}\right].
\end{equation}
This completes the proof of the lemma.
\end{proof}


Now we go back to the  problem $(P_I)$ for a general convex bounded set $\Omega$ in $\mathbb R^N$ and a fixed $l\in(0,1)$ (although we may not write it, the dependence in $l$ will be always implicit). Consider the set of supersolutions as defined in \eqref{supersolutions}. Since $\BB_{\lambda_1} \subset \BB_{\lambda_2}$ for $\lambda_1 > \lambda_2$, we can define the Bernoulli constant for our problem as
\begin{align}\label{lambda-max}
\lambda_{\Omega, max} := \sup \{ \lambda >0 \; | \; \BB_\lambda \; \textrm{ is not empty} \}.
\end{align}

A possible bound from below for $\lambda_{\Omega, max} $ is given by the next lemma.

\begin{lemma}\label{lemma_ball_supers}
Let $l\in(0,1)$ and $\Omega$ be any open bounded convex set, then
$$\lambda_{\Omega, max}  \geq \lambda_{max},$$
where  $\lambda_{max}$ is defined in \eqref{Dmax} if $p=N$ and \eqref{lambda-max2} if $p\neq N$, and $R$ is  the radius of the largest ball inscribed in $\Omega$.
\end{lemma}

\begin{proof}
According to Theorem \ref{thm-interior} and the definition in \eqref{lambda-max}, problem $ ({\text{P}}_I)$ is solvable for given constants $\lambda>0$, $l\in(0,1)$, provided that  $\BB_{\lambda}\neq\emptyset$.
Following the discussion above, for any $ \lambda \le \Lambda(r_{max} )$ there exists a unique constant $\;r_{\lambda,2}$, with $r_{max} \le  r_{\lambda,2}<R$ such that the solution of the $p$-capacitary problem $u_{r_{\lambda,2}}$ in $B_R(0)\setminus \overline{B_{r_{\lambda,2}}(0)}$ satisfies the distance property:
\begin{align*}
 \dd( x,\{u_{r_{\lambda,2}}=l\}) =\lambda  \quad\textrm{for all}\; x\in\partial B_{r_{\lambda,2}}(0).
 \end{align*}
 Note that if $ \lambda = D(r_{max} )$ then $r_{max} =  r_{\lambda,2}$.

  For any  $ \lambda \le D(r_{max} )$ consider now the solution $\bar{u}$ of the $p$-capacitary problem in $\Omega \setminus \overline{B_{r_{\lambda,2}}(0)}$. By comparison principle applied to the functions $u_{r_{\lambda,2}}$ and $\bar u$, the (minimum) distance between the sets $\{ \bar{u} =l\}$ and $\partial B_{r_{\lambda,2}}(0)$ is greater or equal $\lambda$. This shows that $\bar{u}$ is a supersolution for $(P_I)$, as stated.
\end{proof}

\begin{cor}
Fix $l\in(0,1)$ and $\Omega$ a convex bounded domain in $\mathbb R^N$. Then, for each $\lambda\leq \lambda_{\Omega,max}$, there exists a supersolution $(u,K)$ such that $K$ belongs to the class $\BB_\lambda$.
\end{cor}

\begin{proof}
Indeed, we have just seen that if $x_0\in\Omega$ and $R$ the radius of the largest ball centered at $x_0$ inscribed in $\Omega$, then for $\lambda\leq \lambda_{\Omega,max}$ the solution $\bar u$ constructed in Lemma \ref{lemma_ball} is a supersolution of the problem and indeed its zero set $K:=B_{r_{\lambda,2}}(0)$ belongs to the class $\BB_\lambda$.
\end{proof}

\begin{rem} Because of the distance property, for any  $K\in\BB_\lambda$, we have that
$$\dd(x, \partial \Omega) \ge \lambda\quad \textrm{for any}\; x\in \partial K,$$
which shows that any element in the class $\BB_\lambda$ is strictly contained in $\Omega$ and thus, non-degenerate.
\end{rem}

We have just shown  that for any $\lambda \le \lambda_{\Omega, max} $ the set $\BB_{\lambda}$ is not empty. Obviously it holds that $\BB_{\lambda_2} \subset \BB_{\lambda_1}$ if $\lambda_1\le \lambda_2$. We first define the \emph{maximal set} as
\begin{align}\label{maxK}
\KK_\lambda :=\overline{ \mathscr{C}\lp {\bigcup}_{ K\in \BB_\lambda} K \rp},
\end{align}
and will prove that $\KK_\lambda$ satisfies the distance problem $(P_I)$ with constants $(\lambda,l)$. Here $ \overline{ \mathscr{C}(X)}$ denotes the closed convex hull of the set $X$.

It is an interesting open question to see if for general convex domains $\Omega$ the set of all solutions of  problem $(P_I)$ follows the same structure as the ball case, i.e. {\em{parabolic, elliptic}} and {\em{hyperbolic}} case; in particular to prove that the family of {\em{maximal}} convex solutions defined above is an {\em{elliptic}} (i.e. decreasing and continuous) family of solutions and that there exists a unique solution corresponding to the value $\lambda = \lambda_{\Omega,max}$.
These questions have been already addressed for the classical Bernoulli problem and proven to be true for the linear case $p=2$, \cite{A89, Card_Tah02} and for $p\neq 2$ more recently in \cite{BS09}.\\


Before we prove Theorem \ref{thm-interior}, let us show a couple of preparatory lemmas:

\begin{lemma} Let $K_1$ and $K_2$ be two elements of $\BB_\lambda$. It holds that
$$ K^* \in \BB_\lambda,$$
where $K^*$ is the convex hull of the set $K_1 \cup K_2$.
\end{lemma}

\begin{proof}
Let $u^*$ be the solution of the $p$-capacitary problem in $\Omega \setminus K^* $.  Since the set $K^*$ is strictly contained in $\Omega$, the level set $\{u^*=l\}$ is well defined.
Let $x^0\in \{u^*=l\}$ and $x^1\in\partial K^*$ be the points where the minimum distance between  the sets $\partial K^*$ and $\{u^*=l\}$ is attained. We would like to show that $\dd(x^0,x^1)\geq \lambda.$

First, by application of the maximum principle we can quickly conclude that
$$u^*(x)  \le \min\{u_{K_1 }(x),u_{K_2 }(x)\}, \quad \textrm{for each} \; x\in \Omega,$$
which, in particular, implies that
$$l=u^*(x^0)\leq \min\{u_{K_1}(x^0),u_{K_2}(x^0)\}.$$
Because, by initial hypothesis, $u_{K_1}$ and $u_{K_2}$ are supersolutions of the problem $(P_I)$ in the sense of \eqref{supersolutions}, one knows right away that every point in $\{u^*=l\}$  is far from both boundaries $\partial K_1$ and $\partial K_2$ a distance greater or equal than $\lambda$, and in particular the same is true for $x^0$. We have just shown that if we set $K:=K_1\cup K_2$, then
\begin{equation}\label{formula10}\dd(x^0,\partial K)\geq \lambda.\end{equation}

 On the other hand, we have constructed $K^*$ as the convex hull of $K$. If $x^1$ is an extremal point of $K^*$, then it belongs also to $\partial K$ which implies that $\dd(x^0,x^1)\geq \lambda$ by \eqref{formula10}. If $x^1$ is not an extremal point of $K^*$, it can be written as a linear combination of points $\{y_i\}_{i=1}^n$ that are extremal for $K$. For such points we have already shown that  $\dd(y_i,\{u^*=l\})\geq\lambda$. But because the set $\{u^*=l\}$ is the boundary of a convex set by Proposition \ref{convexity-level-sets} one automatically obtains that $\dd(x^1,x^0) \geq\lambda$ too.
\end{proof}

\begin{lemma}
Let $K_1 \subset K_2\subset\ldots$
be a increasing sequence of convex domains in $\BB_\lambda$, and suppose that the set $K$, defined
 as interior of the closure $\overline{\cup K_n}$, is convex, bounded, open and $\Omega\supset K$. Then $K\in\BB_\lambda$.
\end{lemma}

\begin{proof}
Let $u_n$ be the $p$-capacitary potential of $\Omega\backslash K_n$. By construction, $0\leq u_n\leq 1$, and $\{u_n\}$ is a decreasing sequence.
Then we can easily show that on every compact subset of $\Omega$, $u_n$ converges in $\mathcal C^{1,\alpha}$ norm to a $p$-harmonic function $u$.
 We first prove that $u$ is precisely the solution to the $p$-capacitary problem in $\Omega\backslash K$, denoted by $u_K$. Clearly, $u\equiv 1$
 on $\partial\Omega$. We also need to check that $u\equiv 0$ in $K$. For that, let $x\in K\setminus \overline{K_n}$.
Since $K$ is convex, $\Omega\backslash\overline{K_n}$ automatically satisfies the uniform exterior cone condition, and one may use Corollary \ref{cor-holder-estimates} to get a uniform (in $n$) $\mathcal C^{\alpha}$ estimate up to the boundary, hence
$$
0 \leq u_n(x)\leq \dd(x,\overline{K_n})^\alpha\sup \|u_n\|_{\mathcal C^{\alpha}(\overline \Omega\setminus K)} \leq M
\dd(x,\overline{K_n})^\alpha\to 0$$
as $n\to \infty$.
This implies that $u(x)=0$ for all $x\in K$, as desired.

Note that, because we have a decreasing sequence of functions, the set $\{u_n=l\}$ converges to the set where $\{u=l\}$.
However, the boundary of the set $K_n$ may not converge to the set $\{u=0\}$, and this  may happen if the gradient of $u_n$ grows in some neighborhood of $\partial K$.
Hence we need to check that $u(x)>0$ for all $x\in\Omega \setminus \overline K$. Assume that $K \subset\tilde K$ and $u=0$ in $\tilde K$. This is done again by a Harnack chain argument: let $x^0\in\partial\tilde K$ and $P\in\{u=l\}$ be two points such that
$$\dd(x^0,P)\leq \lambda-\eps\quad\mbox{for some }\eps>0.$$
Since $K_n$ satisfies the exterior ball condition, the ball $B_{\eps/2}(x^0)$ is completely outside $K_n$ for all $n$.
Moreover, $u_n(x_n) \to 0$ for all $x\in B_{\eps/2}(x^0)$ as $n\nearrow\infty$.

Cover the segment $\overline{x^0 P}$ by a finite number of balls of radius $\eps/2$, and define $S=\cup_{i=1}^N B_{\eps/2}(x_i)$.
In the set $S$, Harnack inequality applied to the function $u_n$ implies
$$l=u(P)\leq u_n(P)\leq \sup_S u_n(x)\leq C \inf_S u_n(x),$$
where $C$ depends only on $S$. But this is a contradiction, since as $n\nearrow\infty$,
$$l\leq C\inf_S u_n(x)\to 0 .$$
We have shown that, necessarily, $\tilde K= K$,  so that $u>0$ on $\Omega \setminus \overline K$. Then $u=u_K$ is the $p$-capacitor potential
for the set $\Omega\backslash K$. Moreover $K$ belongs to $\BB_\lambda$; the proof follows the same argument as in (\ref{conv-level-set}) with $\Omega_n$ replaced by $K_n$.
\end{proof}

Now we check that the maximal set found in the previous lemma satisfies the distance property. Such maximal set provides the solution for problem $(P_I)$,
and thus the proof of Theorem \ref{thm-interior} is almost complete.

\begin{lemma}
Assume that $\BB_\lambda$ is non-empty. Let $K$  be the maximal set of $\BB_\lambda$ found in the previous lemma and let  $u_K$ be the corresponding solution of
the $p$-capacitor problem for $\Omega\backslash \overline K$. It holds
$$\dd( x,\{ u_K=l \}) =\lambda\quad\mbox{for all}\quad x\in\partial K.$$
\end{lemma}

\begin{proof}
We know that both the sets $\partial K$ and $\{u=l\}$ are convex (the second one is convex by Lemma \ref{convexity-level-sets}).
Then one just needs to understand the distance between two convex sets. By construction, we know that for all $x\in\partial K$.
\begin{equation}\label{equation1}\dd( \{ u_K =l \}, \; x)  \geq \lambda.\end{equation}

Suppose that there is a point $x^0\in \partial K$ such that
\begin{equation*}\dd( \{ u_K =l \}, \; x^0)  = \lambda+\eps.\end{equation*}
Then one can consider the convex set $\tilde K=\{x\in\Omega \,:\,\dd(x,\{u_K=l\})\}=\lambda$.



 Let $\tilde u$ be the $p$-capacitary solution in $\tilde K $; clearly $\tilde u < u_K$ by the strong maximum principle. Then, if $x\in\partial \tilde K$,
$$\dd(x,\{\tilde u=l\})\geq \dd (x,\{u_K=l\})\geq\lambda,$$
which shows that $\tilde K$ is in the class $\BB_\lambda$, which
contradicts the maximality of $K$.
\end{proof}

\emph{Conclusion of the proof of Theorem \ref{thm-interior}:} We have already seen that for any $\lambda<\lambda_{\Omega,{max}}$,
the set $\BB_\lambda$ is nonempty and thus there exists a solution for $(P_I)$. It remains to check that this is so for $\lambda=\lambda_{\Omega,max}$.

Let $\lambda_k$ be  a sequence converging to $\lambda_{\Omega,max}$ from below. We have just shown that for each $\lambda_k$,
 there exists a convex solution $K_k$. One can subtract a subsequence, still denoted by $K_k$, which converges to a convex domain $K$ in the
  Hausdorff topology. Then, by Arzela-Ascoli, the $p$-capacitary potential $u_k$ of $\Omega\backslash  \overline{K_k}$ converges to $u$, the $p$-capacitary potential
  of $\Omega\backslash \overline K$. We have also shown that the distance from the $l$-level set of $u_k$ is at a distance $\lambda_k$,
   then passing to the limit we also see that the same holds for $u$ (similar argument as before).

\begin{rem}\label{remark1}
In contrast to the exterior problem studied before, in the interior discrete Bernoulli problem the free boundary may have corners; consequently $\partial K$ is only $\mathcal C^{0,1}$ and the solution is $\mathcal C^{0,\alpha}$ up to the boundary.
\end{rem}


\subsection{Convergence to the Bernoulli problem}
Before we state the main theorem of the section, let us recall what is known for the interior Bernoulli problem:
\begin{theo}[\cite{HH00}]\label{thm-bernoulli-interior}
Given a convex domain $\Omega$, there exists a constant $\omega_0$ depending on $\Omega$ and $p$ such that for all $\omega\geq \omega_0$ the Bernoulli-type problem
\begin{equation*}
\left\{\begin{split}
&\Delta_p v = 0  \hbox{ in } \Omega\setminus \overline{D},\quad v =0 \textrm{ in } \overline{D}, \quad v =1 \textrm{ in } \partial\Omega,\\
&|\nabla v| =\omega  \quad\textrm{for all}\; x\in\partial D,
\end{split}\right.\leqno ({P_{B_I}})
\end{equation*}
has a solution $(v,D)$ (not necessarily unique). Moreover $D$ is convex. The boundary condition $|\nabla v|=\omega$ is again understood in the sense of \eqref{Bernoulli-condition}.
\end{theo}

\begin{rem}\label{remark2} We remind the reader that the solution $(v,D)$ given in the previous theorem is the one constructed as the extremal supersolution for problem $(P_{B_I})$, although there may be others. The set of supersolutions for each fixed $\omega$ is defined in \cite{HH00} as the class of Lipschitz functions in $\Omega$ such that $v=1$ on $\partial \Omega$, $\Delta_p v\leq 0$ in $\{v>0\}\cap \Omega$ and $|\nabla v|\leq \omega$ on $\partial\{v>0\}\cap\Omega$.
\end{rem}

Now, the main theorem of this section is:
\begin{theo}\label{conv_int}\
Let $\Omega$ be a bounded convex domain in $\mathbb R^N$. Let $\omega>\omega_0$, where $\omega_0$ is the Bernoulli constant for the domain $\Omega$ given in Theorem \ref{thm-bernoulli-interior}. Fix two decreasing sequences  $\{l_n\}$ and $\{\lambda_n\}$ of positive numbers such that $l_n\to 0$, $\lambda_n\to 0$ as $n\to +\infty$ and $l_n =\lambda_n\omega$. Let also $\lambda_{\Omega,max}^n$ be the corresponding quantity \eqref{lambda-max} to each $l_n$. Then:
\begin{itemize}
\item There exists a solution $(v,D)$ for the interior Bernoulli problem $({P_{B_I}})$ with constant $\omega$.
\item There exists a sequence of solutions $(u_n,K_n)$ of problem $(P_I)$ for each $l =l_n $, $\lambda=\lambda_n$.
\item  The sequence of convex sets $\{K_n\}$ is decreasing and as $n\to \infty$ converges in Hausdorff distance to the convex set $D$.
\item In addition, the function $u_n$ converges in $\mathcal C_{loc}^{1,\alpha}(\Omega\backslash \overline D)\cap \mathcal C^{\alpha}(\overline{\Omega\backslash D})$ to $v$.
\end{itemize}
 \end{theo}

The proof of the theorem will be a consequence of the following claims: \\

\emph{Claim 1:} The existence of a solution for the Bernoulli problem with constant $\omega>\omega_0$ is the statement of Theorem \ref{thm-bernoulli-interior}.

\emph{Claim 2:} For each $n$, there exists a solution $(u_n,K_n)$ to the interior problem $(P_I)$ with constants $\lambda=\lambda_n$ and $l=l_n$.

The proof of this claim uses the following lemma on the relation between the solvability of the distance problem and the Bernoulli one:

\begin{lemma} \label{relate-Bernouilli-distance} Given a convex bounded domain $\Omega$ and $l\in(0,1)$, the Bernoulli constant defined in  \eqref{lambda-max} for the distance problem satisfies
\begin{equation}\label{relation}\lambda_{\Omega,max}\geq  \frac{l}{\omega_0},\end{equation}
where $\omega_0$ is the Bernoulli constant from Theorem \ref{thm-bernoulli-interior}.
\end{lemma}

\begin{proof}
As in the proof of Lemma \ref{lemma_ball_supers}, it is enough to show that any solution $(v,D)$ for the Bernoulli problem $ ({\text{P}}_B)$ with constant $\omega_0$ is also a supersolution for $(P_I)$ for $\lambda \le \frac{l}{\omega_0}$. Indeed, the Bernoulli condition for $v$ tells us that $|\nabla v|=\omega_0$ on $\partial D$, which implies that $|\nabla v|\leq \omega_0$ in $\overline{\Omega\backslash D}$ because of Lemma \ref{lemma-gradient-below}. Then, if $x\in\partial D$,
\begin{equation}\label{equation20}\dd(x,\{v=l\})\geq \frac{l}{\sup|\nabla v|}=\frac{l}{\omega_0} \ge \lambda,\end{equation}
which implies $D\in \BB_{\lambda}$.
\end{proof}

Using the lemma for the domain $\Omega$ and the level $l_n$ we may
estimate
$$\lambda_{\Omega,max}^n\geq \frac{l_n}{\omega_0}\geq \frac{l_n}{\omega}=\lambda_n,$$
which tells us that the distance problem $(P_I)$ with constants $(\lambda_n,l_n)$ may be solved
 (Theorem \ref{thm-interior}). This proves Claim 2.\\

\emph{Claim 3:}  $K_n$ is a decreasing sequence of sets. In particular, we have that $u=\sup_n u_n$ and $K=\cap_n K_n$, where $(u_n,K_n)$ is the solution to the distance problem $(P_I)$ with constants $\lambda=\lambda_n$, $l=l_n$.

Fix $m\leq n$, so that $l_n\leq l_m$ and $\lambda_n\leq \lambda_m$. Let us check first that $u_n$ belongs to the set of supersolutions $\BB_m$, defined as
$$
\BB_m := \{ K \subset \Omega \; | \; \inf_{x\in \partial K}  \textrm{dist}(x, \{ u_K = l_m\}) \ge \lambda_m\}.
$$

For that, consider the distance function $d_n$ between two level sets of $u_n$, rescaled by a factor $\omega$,
 so that $d_n$ and $u_n$ have the same boundary values in the set $S_n := \{ 0\le u_n\le l_n \}$. More precisely,
 \begin{equation}\label{rescaled-distance}d_n(x) = l_n[1 - \dd(x, \{u_n=l_n\})/\lambda_n].\end{equation}
Since the level set $\{u_n=l_n\}$ is smooth, the distance function is smooth (a reference for the regularity of distance functions may be found in \cite{Foote:distance-function}). Note that in order to have $d_n$ well defined outside $S_n$, we need to take the distance function to the level set $\dd(x, \{l_n=u_n\})$ with a change of sign once we pass the level set.

Next, since the level sets of $u_n$ are convex surfaces (see Lemma \ref{convexity-level-sets}), the distance function defined in the set $S_n$ is $p$-subharmonic, i.e.,
$$\Delta_p d_n \ge 0.$$
Moreover, since $d_n$ and $u_n$ share the same boundary values on both boundaries of the set $S_n$, we can use Hopf's lemma on the smooth curve $\{ u_n = l_n \}$,
thus obtaining that
$$\omega = |\nabla d_n | \ge | \nabla u_n|,\quad \forall \; x\in \{ u_n = l_n \}.$$
Define now the set  $A_n := \{ l_n\le u_n\le 1 \}$. From Hopf's lemma we can deduce that the function $f_n := d_n - u_n$ is strictly positive in a annulus $M_n$ (a neighborhood of $\{ u_n = l_n \}$) inside $A_n$.  Suppose, by contradiction, that $M_n$ is strictly contained in the neighborhood $A_n$; consider $f_n$ in the domain $\tilde M_n:=S_n \cup M_n$: it holds that $f_n =0$ on $\partial K_n$, $f_n>0$ in  $\{0 < u_n <l_n \}$, $f_n =0$ on $\{u_n = l_n \}$, $f_n >0$ in $M_n$ and $f_n =0$ on $ \partial M_n$. But this is a contradiction since $f_n$ is super-harmonic and should reach its maximum value at the boundary of $\tilde M_n$. Hence $M_n = A_n$.
Then in the set $A_n$ one has that $u_n \leq d_n$. As a consequence,  for $m\leq n$, i.e., $l_m\geq l_n$, and $x\in\partial K_n=\{u_n=0\}$,
$$\dd(x,\{u_n=l_m\})\geq \dd(x,\{d_n=l_m\})=\lambda_m,$$
and $u_n\in\BB_m$ as claimed. In particular, since $u_m$ is constructed as the extremal supersolution in the set $\BB_m$, we have that $u_m\leq u_n$, which in particular implies that $K_n\subset K_m$. This is, $K_n\to K$ is a decreasing sequence of sets, and the claim is proved.\\

\emph{Claim 4: } Let  $ K : = \cap_n K_n$. Then $K \supseteq D$ and is convex. In particular, it is not degenerate.

 Following the steps of Lemma \ref{relate-Bernouilli-distance}, one can show that the solution $(v,D)$ of problem $({P_{B_I}})$ is a supersolution to the $(\lambda_n,l_n)$ distance problem. Hence $K_n\supseteq D$ for any $n\in \N$.

\emph{Claim 5}:  There exists a function $u$ such that $u_n \to u$ in the space $\mathcal C_{loc}^{1,\alpha}(\Omega\backslash \overline K)\cap C^{\alpha}(\Omega\backslash K)$.

\emph{Claim 6: } In addition, $u =0$ on $\partial K$ and $\partial K\in \mathcal C^{0,\alpha}$, so that $u$ is the $p$-capacitary potential in the annulus $\Omega\backslash K$.

Claims 5 and 6 follow similarly as the exterior case as a consequence of uniform $\mathcal C^\alpha$ estimates up to the boundary for a convex annulus $\Omega\backslash K$, so we will not make further comment.\\

Now we remark that the the regularity of the boundary $\partial K$ is slightly worse than in the exterior case (see Remark \ref{remark1}), indeed it is only $\mathcal C^{0,1}$, and we may not have gradient bounds up to the boundary. However, we still have from the proof of Claim 3 that in the set $A_n := \{ l_n\leq u_n< 1\}$ it holds that
\begin{equation}\label{gradient-bound-1}
|\nabla u_n| \le \omega.
\end{equation}
Now we may pass to the limit $n\to \infty$. It is easy to see that $|\nabla u_n|$ converges uniformly to $|\nabla u|$ on compact sets inside $\Omega\backslash \overline K$. As a consequence of \eqref{gradient-bound-1} we get the following claim:

\emph{Claim 7: } $|\nabla u| \leq \omega$ on $\partial K$. From Lemma \ref{lemma-gradient-below} we also get that $|\nabla u|\leq \omega$ in $\overline{\Omega\backslash K}$.\\

To complete the proof of Theorem \ref{conv_int} it suffices to show that $D=K$ and $u=v$.  First, in Claim 4 we have shown that $D\subset K$. On the other hand, in Claim 7 we have shown that $|\nabla u|\leq \omega$ on $\partial K$. This in particular implies that $(u,K)$ is a supersolution for the Bernoulli problem $(P_{B_I})$ in the sense of Remark  \ref{remark2}. As a consequence, one gets that $K\subset D$ (see \cite{HH00}), as desired.  \\

\begin{rem}
If $\Omega$ is a ball, relation \eqref{relation} can be written explicitly. Indeed for $\lambda_{max}$ in Lemma \ref{lemma_ball} via the limit
\begin{equation*}
\lim_{l\to 0} \frac{l}{\lambda_{max}}=\left\{\begin{split}
e/R\quad&\mbox{if}\quad N=p, \\
\frac{1}{R}\left(\frac{p-1}{N-1}\right)^{\frac{N-1}{p-N}}\quad&\mbox{if}\quad N\neq p,
\end{split}\right.
\end{equation*}
one recovers the usual Bernoulli constant for the ball calculated in \cite{HH00}.
\end{rem}

\section{The star-shaped case}\label{exterior-star}

A domain $U\subset \mathbb R^N$ is called star-shaped with respect to the point $x_0\in U$ if for all $x\in U$ the segment joining $x$ and $x_0$ is contained in $U$. Moreover the set $U$ is called star-shaped with respect to the  ball $B\subset U$ if $U$ is star-shaped with respect to every point $x_0\in B$.

Let $U$ be a star-shaped domain with respect to the point $x_0$. For $\theta\in\mathbb S^{N-1}$, where $\mathbb S^{N-1}$ is the unit sphere, set
$$\phi(\theta)=\sup\{\rho>0 : x_0+\rho \theta\in U\}.$$
Then we can write
$$U=\{x_0=\rho\theta : \theta\in\mathbb S^{N-1}, 0\leq \rho<\phi(\theta)\}.$$
It is well known \cite{Burenkov} that:

\begin{lemma}
Let $U$ be a bounded domain in $\mathbb R^N$ which is star-shaped with respect to the point $x_0\in U$. Then it is star-shaped with respect to  a ball centered at $x_0$ if and only if the function $\phi$ satisfies the Lipschitz condition on $\mathbb S^{N-1}$.
\end{lemma}

We call such domains star-shaped Lipschitz. It is clear that star-shaped Lipschitz sets satisfy both the interior and exterior cone condition. One may also consider star-shaped Lipschitz rings: they are of the form $U=\Omega\backslash \overline K$, where $\Omega$ and $K$ are star-shaped Lipschitz with respect to the same center.

Given a bounded, star-shaped Lipschitz domain $K\subset\mathbb R^N$, we consider the exterior distance problem
\begin{equation*}
\left\{
\begin{split}
&\Delta_p u = 0  \hbox{ in } \Omega\setminus\overline{K}, \quad u =1 \textrm{ in }\overline{K}, \quad u =0 \textrm{ on } \partial\Omega, \\
&\dd( x,\Gamma_l) =\lambda  \quad\textrm{for all}\; x\in\partial \Omega,
\end{split}\right.\leqno (P_*)
\end{equation*}
for some function $u$ and a star-shaped domain $\Omega\subset \mathbb R^N$, with the same center as $K$. Our main result in this section is stated in the following theorem:

\begin{theo}\label{thm-star-shaped}
Problem $(P_*)$ has a unique solution $(u,\Omega)$  with $\partial \Omega \in \mathcal C^{0,\alpha}$ and $\mathcal C^{1,1}$ from inside.
\end{theo}

Before we give the proof of the theorem, let us briefly recall what is known for the (exterior) classical Bernoulli problem $({P_B})$ when $K$ star-shaped with respect to a ball. For the case $p=2$  Kawohl proved in \cite{Kawohl:starshaped} existence and uniqueness of a (variational) solution via rearrangement method. In addition he proved that such solution has star-shaped Lipschitz level sets.
Also Acker-Meyer \cite{Acker-Meyer:free-boundary-problem} showed that, still for $p=2$, the solution is regular and that $\partial \Omega \in \mathcal C^\infty$ is also starlike with respect to the same ball.

Now we are ready for the proof of Theorem \ref{thm-star-shaped}: it follows the lines of the problem in the convex case, but some additional ingredients are needed to compensate for the lack of convexity. First, note that uniqueness was already shown in Lemma \ref{lemma_uniq}. 

Note that the regularity results of Theorem \ref{thm-classical-regularity} and Lemma \ref{lemma-holder-estimate} still hold since they only depend on the Lipschitz constant of the boundary of the domain. Moreover, Lemma \ref{lem-gradient-below} is true for star-shaped rings as it was shown in \cite{Lewis-Nystrom:starshaped}: indeed, a  solution of the $p$-Laplacian equation is smooth, $\mathcal C^{1,\alpha}$, and satisfies a type of Harnack inequality, from which a bound from below for the gradient follows. Moreover, we have the following proposition:

\begin{prop}[\cite{Kawohl:starshaped}]\label{starshaped-level-sets}
Let $l\in [0,1)$. Given a (Lipschitz) star-shaped ring $\Omega\backslash \overline K$ and its $p$-capacitary function $u_\Omega$ described as above, the level sets $\{u_\Omega>l\}$ are (Lipschitz) star-shaped.
\end{prop}

Since, by hypothesis, the initial domain $K$ is star-shaped Lipschitz with respect to some point, then there exists some small ball $B_\delta$ centered at that point such that $K$ is star-shaped with respect to that ball. Consider the class
\begin{align*}\CC_*:= \{ \Omega \;{\mbox{open, bounded, star-shaped with respect to $B_\delta$}}, \; \Omega \supset \overline K \},\end{align*}
and define the set of sub- and super-solutions as follows:
\begin{equation*}\begin{array}
{ll}
\AA &= \{ \Omega \in\CC_* \; | \;   \sup\limits_{x\in\partial\Omega}\dd( x,\Gamma_l) \le \lambda \},\\
\AA_0 &= \{ \Omega \in \CC_* \; | \;   \inf\limits_{x\in\partial\Omega}\dd( x,\Gamma_l) <\lambda \}, \\
\BB &= \{ \Omega \in \CC_* \; | \;   \inf\limits_{x\in\partial\Omega}\dd( x,\Gamma_l) \ge \lambda \}.
\end{array}
\end{equation*}

The construction of super- and sub-solutions in Lemma \ref{lemma-sub-supersolution} is valid also for $K$ star-shaped Lipschitz domains; this implies that the sets $\AA_0$ and $\BB$ are nonempty. Lemma \ref{lemma-intersection} holds as well.
Consider now a set $\Omega_0\in\AA_0$ and set $\Omega\subset\mathbb R^N$ to be the intersection of all domains in the class $\BB$ containing $\Omega_0$.  By construction, $\Omega$ is star-shaped with respect to the same ball $B_\delta$. Let $u$ to be the $p$-capacitary potential of $\Omega\backslash \overline K$. Then one may conclude as in Lemma \ref{lemma-convergence-exterior} that $\Omega\in\BB$; the only tricky point is to check that there exists a uniform H\"older estimate in equation \eqref{u_Null}. But this is true because $\mathcal C^\alpha$-estimates up to the boundary from Corollary \ref{cor-holder-estimates} only depend on the Lipschitz constant of the boundary.
Moreover, we know by construction that
$$\inf\limits_{x\in\partial\Omega}\dd( x,\Gamma_l) \ge \lambda.$$
This implies that the equivalent of Lemma \ref{lemma-convergence-exterior} and Corollary \ref{cor-exists-minimal} for star-shaped domains are proven.

Finally, one needs to check that the function $u$ is indeed a solution for the distance problem $(P_*)$, i.e.
\begin{align}\label{ssdistance}
\dd( x,\Gamma_l) = \lambda\quad\mbox{for all }x\in\partial\Omega.
\end{align}
This follows from a simple argument: let $\Omega$ be the extremal set as constructed above, and set $\mathcal L:=\{u>l\}$. Note that $\mathcal L$ is a star-shaped annulus. By contradiction, assume that there exists $x_0\in\partial\Omega$ such that $\dd(x_0,\partial\mathcal L)>\lambda$. Define $\tilde\Omega$ as
$$\tilde\Omega:=\{x\in\Omega : \dd(x,\mathcal L)<\lambda\},$$
 and let $\tilde u$ be the $p$-capacitary potential of $\tilde\Omega\backslash\overline K$. By hypothesis, $\tilde \Omega\subsetneqq \Omega$. Comparison principle implies that $\tilde u\leq u$ and consequently
$$
\dd(x,\{\tilde u=l\})\geq \dd(x,\{u=l\})\geq\lambda,\quad \textrm{for  $x\in\partial\tilde\Omega$}.
$$

Since $\partial \tilde\Omega$ has a fixed distance $\lambda$ from $\Gamma_l$, and $\Gamma_l $ is a star-shaped surface, we conjecture that $\tilde\Omega$ is also star-shaped with respect to the ball $B_\delta$.
We proceed by contradiction: suppose that $\tilde\Omega$ is not star-shaped with respect to a point in $B_\delta$. Without loss of generality we can assume this point to be the origin $x=0$. Note that both sets $\mathcal L:=\{ l\le u \le 1\}$ and $\tilde\Omega$ are smooth. Then there exists a point $z\in\partial\tilde\Omega$ such that the tangent plane $\Pi$ to $\partial\tilde\Omega$ at $z$ passes through the origin.

 By construction the point $z$ lies at a distance $\lambda$ from $\Gamma_l$; in particular there exists $x_0\in\Gamma_l$ where this distance is attained. Since all the sets here are smooth, we know that the line joining $x_0$ to $z$ is normal to the plane $\Pi$. Then one may just work with the intersection of $\tilde\Omega$ with the plane $\Pi_1$, where $\Pi_1$ is the plane that contains the origin, the point $z$ and the normal vector $\vec\nu$ (see Figure \ref{mmm}).

 Consider now the line $s$ passing through the points $0$ and $z$ and the line $t$ passing through $0$ and $x_0$. Since the set $\LL$ is star-shaped with respect to $0$, the curve $\Gamma_l\cap \Pi_1$ has to cross the line $t$ exactly at $x_0$, and after crossing, it lies on the half-plane bounded by $t$ that does not contain the point $z$. Hence there exist some points in $\tilde \Omega\cap \Pi_1$ which are on the half-plane bounded by $s$ (which does not contain $x_0$) that lie at a distance strictly greater than $\lambda$ from $\Gamma_l$. But this contradicts the definition of $\tilde\Omega$.

 Hence $\tilde\Omega$ is star-shaped and belongs to $\BB$; by the minimality of $\Omega$ if follows that $\tilde\Omega = \Omega$ and \eqref{ssdistance} is proven.

\begin{figure}[h]
\includegraphics[width=270pt,height=150pt]{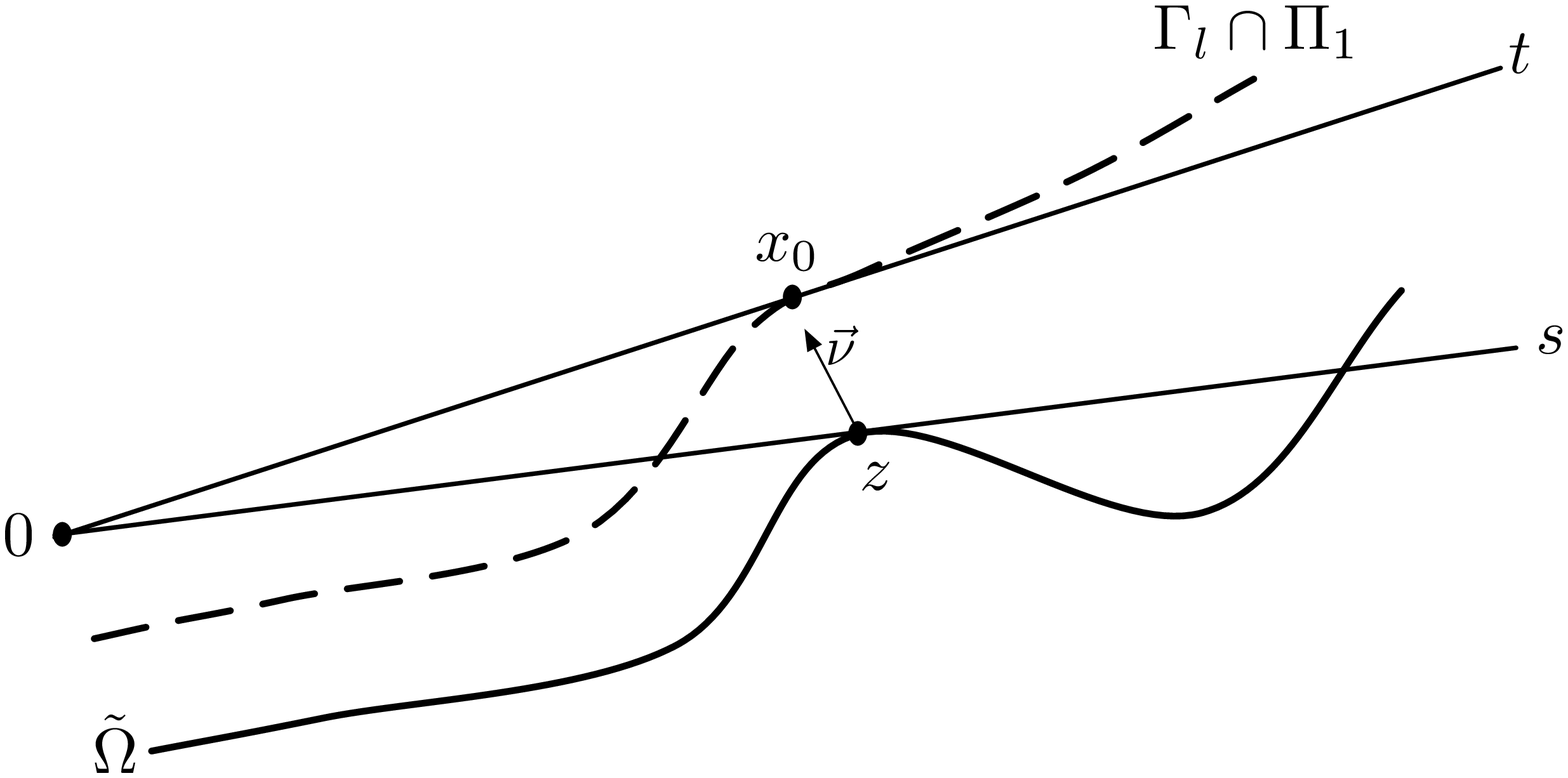}
\caption{Star-shaped domain}
\label{mmm}
\end{figure}

For the regularity, just note that the distance property implies that $\partial\Omega$ is $\mathcal C^{1,1}$ from inside, but further regularity is unknown. The proof of Theorem \ref{thm-star-shaped} is thus completed.


\section{Generalizations}\label{sec:generalization}

\subsection{Non-constant distance property - convex case}\label{subsection:nonconstant}

The next question one may ask is what happens when we replace the constant $\lambda$ in the distance property by a non-constant function $\lambda(x)$. For the Bernoulli problem, this problem was considered in \cite{HH02}. We concentrate on the exterior problem here; the interior problem follows the same line with appropriate modifications.

\begin{theo}
 Let $K\subset \R^N$, $N\ge 2$, be a convex open bounded domain, $l\in(0,1)$ and  $\lambda(x)$ a continuous function satisfying $0<c_0\leq \lambda(x)\leq c_1$ in $\mathbb R^{N}$.  Suppose moreover that $\lambda(x)$ is locally concave in $\mathbb R^N\backslash K$ (i.e., concave on each line segment contained in the set $\mathbb R^N\backslash K$).

 Then there exist a function $u\in\mathcal C^{1,\alpha}(\overline\Omega)$ and a convex open bounded domain $\Omega\subset\mathbb R^N$, $\Omega \supset \overline K$, $\partial\Omega\in \mathcal C^2$ which solve the following problem:
\begin{equation*}\left\{
\begin{split}
&\Delta_p u = 0\hbox{ in } \Omega\setminus\overline{K}, \quad u =1 \textrm{ in } \overline{K}, \quad u =0  \textrm{ on }\partial\Omega, \\
&\dd( x,\Gamma_l) =\lambda(x)  \quad\textrm{for all}\; x\in\partial \Omega.
\end{split}\right.\leqno (P'_E)
\end{equation*}
Moreover if the function
\begin{equation}\label{hypothesis-uniqueness}
t\in[0,1]\mapsto t\lambda((x-z_0)/t)
\end{equation}
is strictly increasing for all $x\,\in\mathbb R^N$ for some $z_0\in K$, then $\Omega$ is unique.
\end{theo}

\begin{proof}

The proof of the above theorem is analogous to one of Theorem \ref{thm-exterior}, with small modifications that we indicate below. First, the sets $\AA$, $\BB$ and $\AA_0$ are defined as in Section \ref{section-exterior}. Super and subsolutions are constructed as in Lemma \ref{lemma-sub-supersolution} by choosing, for the supersolution, $R$ big enough such that
$$
\dd(x, \{u_R =l\}) \ge c_1
$$
and, for the subsolution, $\varepsilon$ small enough such that
$$
\frac{\varepsilon \; l}{\alpha} \le c_0.
$$
Next, Lemma \ref{lemma-intersection} and \ref{lemma-convergence-exterior} follow similarly. The existence of a neighborhood of radius $\varepsilon$ contained in each $\Omega_k$ for all $k$ where one can apply Harnack inequality is a consequence of the fact that $\lambda(x)$ is a continuous function.

Corollary \ref{cor-exists-minimal} holds as well for $\lambda = \lambda(x)$ and the proof follows exactly the one for $\lambda = const.$ Hence we obtain $(\Omega,u)$ minimal in the class $\BB$. $\Omega$ is convex by construction, and thus, $\partial\Omega$ is Lipschitz.

It remains to check that this solution satisfies the distance property, i.e.,
$$\dd(x,\{u=l\})=\lambda(x)\quad \mbox{for all }x\in\partial\Omega.$$
Contrary to the constant case, convexity for $\Omega$ is a crucial hypothesis. By contradiction, assume that the above is not true, i.e. there exists a point $x_0\in\partial\Omega$ such that
$$\dd(x_0,\{u=l\})>\lambda(x_0).$$
By continuity one can find a neighborhood $V \ni x_0$ such that $\dd(x,\{u=l\})>\lambda(x)$ for all $x\in V$. Let $E_\Omega$ be the set of extremal points of $\Omega$. If $x_0\in E_\Omega$, then one may find a convex set $\tilde\Omega$ such that $\Omega\backslash V\varsubsetneq\tilde\Omega\varsubsetneq \Omega$ and $x_0\not\in \tilde\Omega$. Let $\tilde u$ be the $p$-capacitary function of $\tilde \Omega$. Comparison principle implies that $\tilde u\leq u$; then for all $x\in\partial\tilde\Omega$,
$$\dd(x,\{\tilde u=l\})\geq \dd(x,\{u=l\})\geq\lambda(x).$$
We have just shown that $\tilde\Omega\in\BB$. Contradiction with the minimality of $\Omega$.

If $x_0\not\in E_\Omega$, then $x_0$ can be written as a linear combination of extremal points, i.e.,
\begin{equation}\label{non-extremal-point}x_0=\sum_{i=1}^n t_n x_n,\quad \mbox{where } \sum_{i=1}^n t_i=1, x_i\in \overline{E_\Omega}, i=1,\ldots,n.\end{equation}
By hypothesis $\dd(x_i,\{u=l\})=\lambda(x_i)$ for every $i=1,\ldots, n$ and
$$\dd(x_0,\{u=l\})\leq \sum_{i=1}^n t_i \dd(x_i,\{u=l\})=\sum_{i=1}^n t_i \lambda(x_i)\leq\lambda(x_0),$$
where we have used the local concavity assumption on $\lambda$ for the last inequality.\\

Next, to show that that the solution is unique we follow the same rescaling method as in Lemma \ref{lemma_uniq}. Suppose, by contradiction, that $(P'_E)$ has two solutions $(u_1,\Omega_1)$, $(u_2,\Omega_2)$. Without loss of generality assume that in hypothesis \eqref{hypothesis-uniqueness} we have that $z_0=0\in K$. Rescale $u_2^\varepsilon(x)=u_2(x/\varepsilon)$. Then $\Omega^\varepsilon_2\subset \Omega_1$ and one may consider the biggest $\varepsilon(<1)$ such that $\Omega_2^\varepsilon$ touches $\Omega_1$ from inside. Let $x_0^\varepsilon\in \partial \Omega_2^\varepsilon \cap \Omega_1$ and $x_0=x_0^\varepsilon/\varepsilon$. It holds
\begin{equation}\label{formula20}
\dd(x_0^\varepsilon,\{u_2^\varepsilon=l\})=\varepsilon \dd(x_0,\{u_2=l\})=\varepsilon \lambda(x_0)=\varepsilon \lambda( x_0^\varepsilon/\varepsilon)<\lambda(x_0^\varepsilon),
\end{equation}
thanks to \eqref{hypothesis-uniqueness} for the last inequality. On the other hand, by comparison principle, we easily have that $u_2^\varepsilon\leq u_1$; in particular,
\begin{equation}\label{formula21}
\dd(x^\varepsilon_0,\{u_2^\varepsilon=l\})\geq \dd(x^\varepsilon_0,\{u_1=l\})=\lambda(x^\varepsilon_0).
\end{equation}
Then equations \eqref{formula20} and \eqref{formula21} give the desired contradiction.

We look now at the regularity: since $\Omega$ is convex the surface $\{u=l\}$ is analytic by Lemma \ref{convexity-level-sets}. Then $\partial\Omega$ is the set of points at distance $\lambda(x)$ from this level set, that will have principal curvatures given by the formula
$$\frac{\mu_i}{1+ \lambda(x)\mu_i}.$$
If $\lambda$ is continuous, then $\partial\Omega$ will have continuous principal curvatures, which implies that $\partial\Omega\in\mathcal C^2$. Higher regularity on $\lambda$ will imply higher regularity for $\partial\Omega$. The proof of the theorem is thus completed.

\end{proof}

\subsection{The distance problem with mean curvature condition}

 We shall now consider the case where $\lambda=\lambda(x,\kappa(x))$, with $\kappa(x)$ being the mean curvature of $\partial\Omega$ at the point $x$. The standard Bernoulli problem for harmonic functions (case $p=2$) with non-constant gradient condition depending on the mean curvature has been studied in \cite{Athanasopoulos-Caffarelli-Kenig-Salsa,Mazzone,Mikayelyan-Shahgholian} through a variational formulation for the functional
$$\int_{D} |\nabla u|^2\,dx + \Per(\{v>0\}).$$
Here we follow the sub/supersolution approach.

We first recall the definition of mean curvature in the viscosity sense:

\begin{defin}
For $U$ a bounded convex domain, one may define the {\em{(interior)}} mean curvature  of $\partial U$ in the viscosity sense as follows: assume that $0\in\partial U$ and that the interior normal $\nu$ to $\partial\Omega$ at the origin is in the direction of the $e$-axis. We define
$$\kappa(\partial U)(0):=\inf_{A\in\mathfrak{U}} \kappa (S_A)(0),$$
where $S_A=\{(x,e)\, :\,e=\langle Ax,x\rangle\}$ and $\mathfrak{U}$ is the set of all symmetric matrices $A$ such that the set $S_A$ (the graph of a quadratic polynomial) locally touches $\partial U$ at $0$ from inside.

If the set $\mathfrak U$ is empty (i.e., no paraboloid touches $\partial U$ at $0$ from inside), we define
$$\kappa(\partial U)(0)=\infty.$$
\end{defin}

We also define the {\em{exterior}} mean curvature

\begin{defin}
Let $U$ be as in the previous definition:
We define the  {\em{exterior}} mean curvature $\underline \kappa(\partial U)(0)$ as

$$\underline \kappa(\partial U)(0):=\sup_{A\in\mathfrak{U}} \kappa (S_A)(0),$$
where $S_A=\{(x,e)\, :\,e=\langle Ax,x\rangle\}$ and $\mathfrak{U}$ is the set of all symmetric matrices $A$ such that the set $S_A$ (the graph of a quadratic polynomial) locally touches $\partial U$ at $0$ from outside.

If the graph of any quadratic polynomial locally touches $\partial U$ at $x=0$ from outside, we say that
$$\underline\kappa(\partial U)(0)=\infty.$$
\end{defin}

Clearly for any point $x_0 \in \partial U$ where $\partial U \in \mathcal C^2$ in a neighborhood of $x_0$ it holds that $\underline\kappa (\partial U)(x_0)= \kappa(\partial U)(x_0)$. \\


In this section we study the exterior discrete Bernoulli problem $(P'_E)$ with distance condition depending on the mean curvature. More precisely, given a convex bounded domain $K$ with $\partial K\in \mathcal C^{1,1}$ we prescribe the condition
 \begin{equation}\label{distance-mean-curvature}\dd(x,\Gamma_l) \; = \frac{\tilde l}{\kappa(\partial \Omega)(x)}=:\lambda_{\kappa,\tilde l}(x),\quad x\in\partial\Omega,\end{equation}
 with $\tilde l\in (0,1)$. We do not include $\tilde l = 1$ since it is not an admissible value for the case when $K$ is a ball, as explained in Lemma \ref{lemma:curvature_super}.

The main result is summarized in the following theorem.

\begin{theo}\label{theo-curvature}
 Let $K\subset \R^N$, $N\ge 2$, be a convex open bounded set and $l,\tilde l \in(0,1)$. There exist a function $u\in\mathcal C^{1,\alpha}(\Omega)$ and a convex open bounded domain $\Omega\subset\mathbb R^N$, $\Omega\supset \overline K$, with $\partial\Omega\in \mathcal C^{0,1}$ which solve the following problem:

\begin{equation*}\left\{
\begin{split}
&\Delta_p u = 0\hbox{ in } \Omega\setminus\overline{K}, \quad u =1 \textrm{ in } \overline{K}, \quad u =0  \textrm{ on }\partial\Omega, \\
&\dd( x,\Gamma_l) \geq\lambda_{\kappa,\tilde l}(x)  \quad\textrm{for all}\; x\in\partial\Omega,\\
&\dd( x,\Gamma_l) =\lambda_{\kappa,\tilde l}(x)  \quad\textrm{for all}\; x\in\partial \Omega_\kappa,
\end{split}\right.\leqno (P'_K)
\end{equation*}
where $\lambda_{\kappa,\tilde l}(x)$ is defined in (\ref{distance-mean-curvature}) and $\partial \Omega_\kappa$ is the set of all points $x\in \partial\Omega$ for which at least one of the following assumptions is satisfied:
\begin{enumerate}
\item[\emph{$(A1)$}] $x$ is an extremal point of $\Omega$ and $\partial\Omega$ is $\mathcal C^2$-smooth at $x$,
\item[\emph{$(A2)$}] $x$ is not an extremal point, but it can be written as a linear combination of extremal points satisfying \emph{$(A1)$}.
\end{enumerate}

Moreover the set $\partial\Omega\in \mathcal C^{0,1} $ does not contain points $x$ of the form
\begin{enumerate}
\item[\emph{$(A3)$}] $x$ is an extremal point of $\Omega$ and $\kappa(\partial\Omega)(x)=\underline \kappa(\partial\Omega)(x)=\infty$,
\item[\emph{$(A4)$}] $x$ is not an extremal point, but it can be written as a linear combination of extremal points satisfying \emph{$(A3)$}.
\end{enumerate}
\end{theo}

Recall that extremal points are defined in Definition \ref{extr-points}. We first show a technical lemma that will be used later:

\begin{lemma}
Let $\Omega_n$ be a decreasing sequence of convex bounded domains converging in Hausdorff measure to $\Omega$. Let $x_0$ be a point on $\partial\Omega$; it holds
\begin{equation}\label{outer-curvature}
\liminf_{\{x_n\}} \kappa(\partial\Omega_n)(x_n)\leq \kappa(\partial\Omega)(x_0),
\end{equation}
where the $\liminf$ is taken for all sequences $\{x_n\}$ satisfying $x_n\in\partial\Omega_n$ and $x_n\to x_0$.
\end{lemma}

\begin{proof}
First note that if $\kappa(\partial\Omega)(x_0) = +\infty$ inequality \eqref{outer-curvature} holds.
Let $x_0$ be a regular point with curvature $0\leq k_0:=\kappa(\partial\Omega)(x_0) < +\infty$. Without loss of generality assume that $x_0 =0$. For any $\varepsilon>0$, there exists a paraboloid $P_{k_0+\varepsilon}$ with curvature ${\kappa_0+\varepsilon}$ at the point $x=0$ that touches $\partial\Omega$ from inside at $0$ in a neighborhood $U_\varepsilon\ni 0$ (see Figure \ref{mov_curv}).

\begin{figure}[h]
\includegraphics[width=260pt,height=200pt]{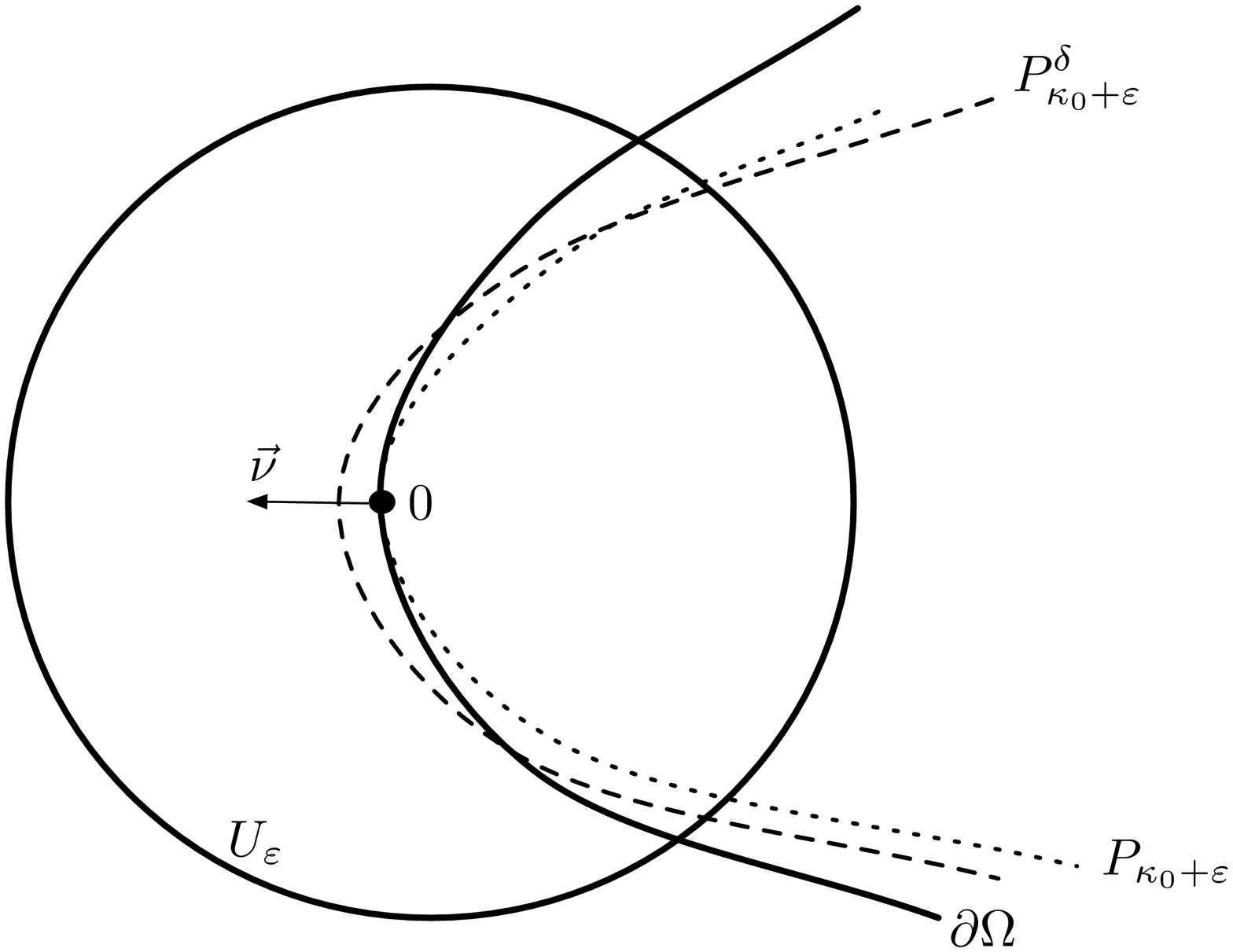}
\caption{}
\label{mov_curv}
\end{figure}

 Shift now the paraboloid $P_{k_{0}+\varepsilon}$ by a small positive distance $\delta$ in the direction of the outward normal $\vec{\nu}$ to $\partial\Omega$ at $0$; denote this shifted ball by  $P_{k_0+\varepsilon}^\delta$.  Clearly, if $\delta$ is small enough, $P_{k_0+\varepsilon}^\delta \cap \partial \Omega \subset U_\varepsilon$. Since $\Omega_n$ are convex sets and $\{\partial\Omega_n \}$ converges to $\partial \Omega$ as $n\nearrow \infty$, there exists at least a set $\Omega_{n_\delta}$, with $n_\delta$ dependent on $\delta$, such that $ \Omega\cup (P_{k_0+\varepsilon}^\delta\cap U_\varepsilon) \subseteq \Omega_{n_\delta}$ and  $\partial \Omega_{n_\delta} \cap \partial P_{k_0+\varepsilon}^\delta = \{x_{n_\delta}\}$; moreover $P_{k_0+\varepsilon}^\delta$ touches tangentially $\partial \Omega_{n_\delta}$ at the point $x_{n_\delta}$. Inequality \eqref{outer-curvature} follows since $\varepsilon$ and $\delta$ can be taken arbitrarily small.
\end{proof}

The supersolutions are given by those $\Omega$ such that
\begin{equation}\label{condition10}\dd(x,\Gamma_l){{\ge }}\;  \lambda_{\kappa,\tilde l}(x),\end{equation}
and hence we define the set of all supersolutions as
$$\BB:=\{\Omega\subset \mathbb R^N,\mbox{ convex, bounded}, \; K\Subset \Omega : u_\Omega \mbox{ satisfies }\eqref{condition10} \}.$$

Note that if $x_0\in\partial\Omega$ is a point where $\kappa(\partial \Omega)(x) = +\infty$ then the distance property \eqref{condition10} does not impose any restriction.
One may define analogously $\AA$ (and $\AA_0$) to be the set of subsolutions (strict subsolutions)   through condition
\begin{equation}\label{condition11}\dd(x,\Gamma_l){{\le }} (<)\; \lambda_{\kappa,\tilde l}(x).
\end{equation}

\begin{lemma}
If $\Omega_1,\Omega_2\in\BB$, then $\Omega_1\cap\Omega_2\in \BB$.
\end{lemma}

\begin{proof}
The proof is a simple consequence of the fact that
$$\dd(x_0,\{u_{\Omega_1\cap\Omega_2}=l\})\geq \max\left\{\dd(x,\{u_{\Omega_1}=l\}), \dd(x_0,\{u_{\Omega_2}=l\})\right\}.$$
\end{proof}

\begin{lemma}\label{lemma:curvature_super}
 The sets $\BB$ and $\AA$ are non empty. Moreover, any supersolution is greater or equal than every admissible subsolution.
 \end{lemma}

\begin{proof}
A possible supersolution for our problem is the function $u_R$
constructed in Lemma \ref{lemma-sub-supersolution}: this depends on
the fact that  there exists $R$ big enough such that
$$
\dd(x,\Gamma_l) \ge \tilde l\,R \quad \forall x\in \partial B_R(0)\quad \textrm{and} \quad \tilde l\in (0,1).
$$
Note that if $K$ is the ball $B_r(0)$ and $\bar u(|x|)$ the capacitary function in the annulus $B_R(0) \setminus B_r(0)$, condition (\ref{formula_used}) becomes
 $$
 \left( 1 - \left(\frac{r}{R}\right)^l\right) \ge \tilde l.
 $$
Clearly there exists some finite value for $R$ (dependent on $\tilde l$) that satisfies the inequality above if $ \tilde l\in(0,1)$ but none if $\tilde l = 1$.

For the subsolution consider the function $u_\varepsilon$ constructed in Lemma \ref{lemma-sub-supersolution}; it holds
$$
\dd(x,\{u_{\varepsilon} =l\})\le \frac{\varepsilon \tilde l}{\alpha}, \quad \mbox{for all }x\in \{u_\varepsilon =0\},
$$
and the desired property \eqref{condition11} is satisfied by choosing
$$\varepsilon \le \frac{\alpha}{ \max_{x\in \{u_R = 1 - \varepsilon\}} \kappa(x)},$$
where $\kappa(x)$ is the mean curvature of the level set $\{u_R = 1 - \varepsilon\}$.
Note that  $\max_{x\in \{u_R = 1 - \varepsilon\}} \kappa(x) < +\infty$ since this level set is a analytic surface (see Proposition \ref{convexity-level-sets}).

Lastly we show that any supersolution is larger than every subsolution. We proceed as in Lemma \ref{lemma-sub-super} by assuming that there exists a subsolution, denoted by $(\Omega_S, u_S)$ which is not everywhere larger than a subsolution, denoted by $(\Omega_s, u_s)$. We rescale the subsolution until $\partial\Omega_s^\varepsilon$ touches $\partial\Omega_S$ at a point $x_0 \in \partial\Omega_s^\varepsilon\cap \partial\Omega_S$. Then
$$
 \kappa(\partial\Omega_s^\varepsilon)(x_0) \ge  \kappa(\partial\Omega_S)(x_0).
$$
Using strong comparison principle we arrive at the following contradiction:
$$
\frac{\tilde l}{\kappa(\partial\Omega_s^\varepsilon)(x_0)}  \ge \dd(x_0 ,\; \{u_s^\eps=l\}) >  \dd(x_0,\;\{u_S=l\}) \ge \frac{\tilde l}{\kappa(\partial\Omega_S)(x_0)}.
 $$
This concludes the proof of the lemma.
\end{proof}

\begin{lemma}\label{lemma-conv-B}
Let $\{\Omega_n\}$ be a decreasing sequence of bounded convex
domains in $\BB$, and define $\Omega:=\mbox{Interior}(\overline{\cap
\Omega_k})$. Then the $p$-capacitary potential of
$\Omega\backslash\overline K$, denoted by $u_\Omega$, satisfies
\eqref{condition10} and thus, $\Omega\in \BB$. Moreover, there
exists a positive constant $c_0$ such that
$\kappa(\partial\Omega)(x)\geq c_0$ for all $x\in\partial\Omega$.
\end{lemma}

\begin{proof}
Since each domains $\Omega_n$ is convex, we can start as in the proof of Lemma \ref{lemma-convergence-exterior}. Let $u_n$ be the $p$-capacitary potential of $\Omega_n\backslash\overline K$. The sequence  $\{u_n\}$ converges in $\mathcal C^{1,\alpha}$ norm to a $p$-harmonic function $u$. Convexity immediately implies that $u\equiv 0$ outside $\Omega$, as it was done in \eqref{u_Null}. In addition it holds that the set $\{u_n=l\}$ converges to the set $\{u=l\}$ and that $\partial\Omega_n$ converges to $\partial\Omega$ in Hausdorff distance.

We show first that $K\Subset \Omega$: the existence of a subsolution (found in Lemma \ref{lemma:curvature_super}) implies, using Lemma \ref{lemma-gradient-above}, that $|\nabla u_n|\leq M$ for a constant $M$ independent of $n$.  Consequently, for any $n\in \N$, it holds
\begin{align}
\label{bound-below-dist}
\dd(x,\{u_n =l\})\geq \frac{l}{M},\quad \forall x\in\partial\Omega_n.
\end{align}
This implies that $K\Subset \Omega$. Unfortunately inequality (\ref{bound-below-dist}) does not give any information on uniform bounds from above for $\kappa(\partial \Omega_n)$.

Instead we show that $\kappa(\partial\Omega)$ is uniformly bounded from below by a positive constant. By construction we have that $\Omega_n\subset \Omega_0$ for every $n$. Then, for every point $x_n\in\partial\Omega_n$,
\begin{align}
\label{bound-above-dist}
\frac{\tilde l}{\kappa(\partial \Omega_n)(x_n)}\leq\dd(x_n, \{u_n=l\}) \le diam(\Omega_0), \quad \forall x_n\in\partial\Omega_n.
\end{align}
In particular $\partial \Omega_n$ cannot have flat parts. Moreover, if we take  a sequence $x_n\in\partial\Omega_n$ such that $x_n\to x_0$ as $n\nearrow \infty$ with $x_0\in\partial\Omega$, from  (\ref{bound-above-dist}) and \eqref{outer-curvature} (maybe after passing to a subsequence) we get
$$\frac{\tilde l}{\kappa(\partial \Omega)(x_0)} \le diam(\Omega_0), \quad \forall x_0\in\partial\Omega,$$
which shows that $\partial \Omega$ cannot have flat parts either.

Two things are now to be proven: $(i)$ the function $u$ is the $p$-capacitary potential in $\Omega \backslash {\overline K}$ and $(ii)$ the function $u$ satisfies \eqref{condition10}.
Suppose that $\tilde\Omega :=\{supp(u)\}\subset \Omega$; then there exists a point $x_0\in \partial\tilde\Omega\backslash \partial\Omega$ with a neighborhood of radius $\varepsilon$, denoted by $B_\varepsilon(x_0)$, such that $\tilde\Omega\cup B_\varepsilon(x_0) \subset \Omega$.
Let $P\in \{u=l\}$ be the closest point of $\{u=l\}$ to $x_0$ and cover the segment $\overline{x_0,P}$ with a finite number of balls if radius $\varepsilon$. Note that this is possible because the segment $\overline{x_0,P}$ has a non-negative finite length whose bounds are given in (\ref{bound-below-dist}) and (\ref{bound-above-dist}). By mimicking the argument in Lemma \ref{lemma-convergence-exterior} one can show that Harnack inequality leads to a contradiction. Hence $(i)$ is proven.

Consider now a point $x_0\in \partial\Omega$ and any sequence $\{x_n\in\partial\Omega_n\}$ such that $x_n\to x_0$ as $n \nearrow\infty$; it holds
$$
\dd(x_n,\{u_n=l\}) \to \dd(x_0,\{u=l\}), \quad \textrm{as}\; n\nearrow\infty.
$$
On the other hand, taking $\limsup_{n\to +\infty}$ on both sides of \eqref{condition10} and using \eqref{outer-curvature} we obtain
\begin{align}\label{problem}
\lim_{n\to +\infty} \dd(x_n,\{u_n=l\}) &\ge \limsup_{n\to +\infty} \frac{\tilde l}{\kappa(\partial \Omega_n)(x_n)} \geq  \frac{\tilde l}{{\kappa}(\partial \Omega)(x_0)}. \nonumber
\end{align}
The proof is completed.
\end{proof}

{\em{\bf Proof of Theorem \ref{theo-curvature}}}.
It remains to prove that the $p$-capacitary function $(u_\Omega,\Omega)$ constructed in Lemma \ref{lemma-conv-B} satisfies the distance property
 \begin{equation}\label{rep-distance}
 \dd( x,\Gamma_l) =\lambda_{\kappa,\tilde l}(x),
 \end{equation}
 for every $x$ that satisfies either \emph{$(A1)$} and \emph{$(A2)$} and that $\partial\Omega$ does not have points that satisfy \emph{$(A3)$} nor \emph{$(A4)$}.


Suppose, by contradiction, that there exists an extremal point $x_0\in \partial\Omega$ such that
$$\dd( x_0, \{u_\Omega=l\}) \ge \frac{\tilde l}{ \kappa(\partial\Omega)(x_0)} + \delta,\quad \delta >0.$$
Then there exists a neighborhood $U$ of $x_0$ (depending on $\delta$) such that for any $x\in U\cap \partial\Omega$ it holds
\begin{equation}\label{formula-uffa}
\dd(x,\{u_\Omega=l\}) \ge \frac{\tilde l}{ \kappa(\partial\Omega)(x_0)} + \frac{\delta}{2}.
\end{equation}

We first assume that $\partial\Omega$ is $\mathcal C^2$ in a neighborhood of $x_0$, i.e. assumption \emph{$(A1)$} is satisfied. We apply translations and rotations so that $x_0=0$ and choose coordinates $x=(x',x_n)$ in such a way that $\Omega\subset \{x_n>0\}$ and parameterize the surface $\partial\Omega$ near $x_0$ as a graph $x_n=\varphi(x')$, $x'\in \R^{N-1}$, such that $\varphi(0)=0$, $\nabla \varphi(0)=0$. Defined a new function $\varphi_1(x'):=(1-\varepsilon)\varphi(x')+\varepsilon^2$. For $\varepsilon$ small enough the intersection set $V_\varepsilon:= \{ (x',x_n), \; | \; \varphi(x') = \varphi_1(x')\}$ is strictly contained in the small neighborhood $U$. Consider the new function $\tilde \varphi:=\max\{\varphi,\varphi_1\}$ and define the set
$$\tilde \Omega=(\Omega \cap U\cap \{(x',x_n) \; | \; x_n>\tilde\varphi(x')\}) \cup (\Omega \cap U^c).$$
By construction it holds that $\tilde\Omega\varsubsetneq \Omega$; moreover the mean curvature of $\tilde\Omega$ can be estimated in terms of the mean curvature or $\partial\Omega$ at $x_0$, except maybe at the points in $V_\varepsilon\cap \partial \Omega$, in which it may be infinity. Hence let $\tilde u$ be the $p$-capacitary potential of $\tilde\Omega\backslash \overline K$. Comparison principle and \eqref{formula-uffa} yield, for $\varepsilon$ small enough,
 $$\dd(x_1,\{\tilde u =l\}) \ge \dd(x_1,\{u_\Omega=l\})\ge \frac{\tilde l}{ \kappa(\partial\Omega)(x_0)} + \frac{\delta}{2} \geq \frac{\tilde l}{ \kappa(\partial\Omega)(x_1)},$$
for $x_1\in\partial\tilde\Omega\cap U$. Note that we may obtain the same conclusion of $x_1\in V_\varepsilon$ even if the curvature at those points is ninfinity. The existence of such $\tilde\Omega\varsubsetneq \Omega$ contradicts the minimality of $\Omega$.

Let us assume now that the extremal point $x_0$ satisfies \emph{$(A3)$}.
Inequality (\ref{formula-uffa}) reduces to
$$\dd( x, \{u_\Omega=l\}) \ge \frac{\delta}{2}.$$
We perturb $\Omega$ to $\tilde \Omega$ in the following way: consider, as before, $\partial\Omega$ written as the graph $x_n=\varphi(x')$. By hypothesis, $\overline \kappa(\partial\Omega)(x_0)=\infty$, so for each constant $M>0$, there exists a paraboloid $m(x')$ with curvature equal $M$ at $x_0$ that touches $\partial\Omega$ from outside at $x_0$. Define a new function $\varphi_1(x'):=(1-\varepsilon)m(x')+\varepsilon^2$ and consider $\tilde\varphi:=\max\{\varphi,\varphi_1\}$. Define the set $\tilde\Omega$ as in the previous case; contradiction follows similarly. Hence in such point $x_0$ we have that $\dd(x_0,\{u_\Omega=l\})=0$; but this is not possible because we have uniform gradient bounds for $u_\Omega$ (see Corollary \ref{lemma-gradient-above}).

Next we consider the case of $x_0$ not being an extremal point of $\partial\Omega$: Lemma 12 in \cite{Mikayelyan-Shahgholian} shows that $\frac{1}{\kappa(\partial\Omega)(x)}$ is a locally concave function if $\Omega$ is convex. Then we may proceed as in \eqref{non-extremal-point} and following arguments in order to conclude that property (\ref{rep-distance}) is satisfied for each points of the form  \emph{$(A2)$}  and that $\partial\Omega$ does not contain points of the form  \emph{$(A4)$}.
\qed



\begin{rem}
Open problems:
\begin{itemize}
\item[{\em(i)}] Full regularity of $\partial\Omega$ for $(P_K')$. For the classical Bernoulli problem full regularity of $\partial\Omega$ is obtained via the well known regularity theory for almost minimal surfaces (see \cite{Almgren76}). It is an interesting open question to show that such theory can be applied also for $(P_K')$.
\item[{\em(ii)}]  One may also consider the problem with more general distance functions $\lambda_{\kappa,\tilde l}=\lambda(x,\kappa(x))$ in $(P_K')$.
    \end{itemize}
\end{rem}


\subsection{Two- and Multi-phase problems}

The multi-phase Bernoulli problem (c.f.
\cite{Acker:multilayer,Acker:multilayer2,Laurence-Stredulinsky} for
$p=2$, \cite{AHPS} for general $p$), has many applications in
studying interfaces, for instance. In a similar flavor we formulate
the multi-phase discrete Bernoulli problem. Fix $m\in\mathbb N \cup
\{0\}$ and let $K_1,K_{m+2}$ be two  convex bounded domains such
that $\overline K_1 \subset K_{m+2}$. Moreover consider a sequence
of real numbers $-1\leq a_i\leq 1$, $i=2,\ldots,m+1$ with
$a_i>a_{i+1}$ and continuous functions $g_i:(K_{m+2}\backslash
K_1)\times \mathbb R_+\to \mathbb R_+$, $i=2,\ldots,m+1$. We seek a
sequence $\{K_i: i=2,\ldots,m+1\}$ of convex domains such that
$K_1\Subset K_2\Subset \ldots\Subset K_{m+1}\Subset K_{m+2}$
solve the multi-layer distance problem. Let $u_i$ be the
$p$-capacitary potential of the set $K_{i+1}\backslash \overline
K_i$, i.e., solution of
$$\left\{
\begin{array}{lll}
\Delta_p\; u_i = 0  \;\hbox{ in }\; K_{i+1}\setminus \overline{K_i},  && \\
u_i =a_i \quad \textrm{on}\; \partial K_i, && \\
u_i = a_{i+1}\quad \textrm{on}\; \partial K_{i+1}.&&
\end{array}\right.\leqno (P_M)
$$
The functions $u_i$ must satisfy the following (nonlinear) joining conditions: for $0<l<\min\{|a_i-a_{i+1}|\}$ the level sets
\begin{equation*}
\Gamma^i_{+l}:=\{u_i=a_{i+1}+l\},\quad \Gamma^{i+1}_{-l}:=\{u_{i+1}=a_{i+1}-l\}
\end{equation*}
satisfy
\begin{equation}\label{joining-F}
F_i\lp x,\dd(x,\Gamma^i_{+l}),\dd(x,\Gamma^{i+1}_{-l})\rp=0,\quad \mbox{for all }x\in \partial K_{i+1},
\end{equation}
for given functions $F_i: K\times \mathbb R_+ \times \mathbb R_+\to \mathbb R$, $i=1,\ldots,m$.

We assume that $F_i(x,p,q)$ is a continuous function, and that $F_i(x,p,q)$ is strictly decreasing as a function of the variable $p$ for all $x,q$. Then via the implicit function theorem one can write the joining condition \eqref{joining-F} as
$$\dd(x,\Gamma^i_{+l})=g\lp x,\dd(x,\Gamma_{-l}^{i+1})\rp,\quad \mbox{on } \partial K_{i+1},$$
for some functions $g:K\times \mathbb R_+ \to \mathbb R_+$.

We assume the following initial hypothesis on the functions $g_i$, $i=1,\ldots,m$:
\begin{itemize}
\item[(H1)] $g_i$ is continuous and bounded from below by a positive constant.
\item[(H2)] $g_i$ is non-decreasing with respect to the second argument.
\item[(H3)] The function $x\mapsto g_i\lp x,q(x)\rp$ is a concave function whenever $q(x)$ is a concave function.
\item[(H4)] For any given value $y_0>0$, $i=1,\ldots,m$, there exist constants $0<c_1<c_2$ such that $c_1\leq g_i(x,y)/y\leq c_2$, uniformly for all $x\in K$ and all $y\geq y_0$.
\end{itemize}

Although condition (H4) may seem artificial at first, it is the classical hypothesis that assures convexity for the level sets of the solution to the Bernoulli problem (c.f. \cite{Laurence-Stredulinsky}).  The classical example is
$$g(x,q)=\frac{1}{(a(x)^\alpha+q^\alpha)^{1/\alpha}},$$
for a function $a$ such that $1/a(x)$ is concave. Existence, uniqueness and convexity of solutions for a general $g$ is an open problem. \\

The main results of this sections are in the following two theorems:

\begin{theo}\label{thm:two-phase}
Let $K_1,K_3\subset \mathbb R^N$ be two convex bounded domains such that $K_1\Subset K_3$. Then there exists a convex domain $K_2$ such that
$$K_1\Subset K_2\Subset K_3,\quad \partial K_2\in\mathcal  C^{1,1}, $$
and the $p$-capacitary potentials $u_1$ and $u_2$ of the sets $K_2\backslash \overline {K_1}$ and $K_3\backslash \overline{K_2}$, respectively, i.e., solutions of
\begin{equation}\label{capacitari-2-multi}
\left\{ \begin{array}{ll}
\Delta_p u_1=0 & \mbox{ in }K_2\backslash \overline{K_1},\\
u_1=1  &\mbox{ on }\partial K_1, \\
u_1=0 & \mbox{ on }\partial K_2,
\end{array}
\right.
\quad \left\{\begin{array}{ll}
\Delta_p u_2=0 & \mbox{ in }K_3\backslash \overline{K_2},\\
u_2=0 & \mbox{ on }\partial K_2,\\
u_2=-1  &\mbox{ on }\partial K_3,
\end{array}\right.
\end{equation}
satisfy the nonlinear joining condition
\begin{equation}\label{joining-condition}
\dd(x,\{u_1=l\})=g\lp x,\dd(x,\{u_2=-l\})\rp \quad\mbox{on }\partial K_2.
\end{equation}
\end{theo}

\begin{theo}\label{thm:multi-phase}
The multi-phase version of Theorem \ref{thm:two-phase} holds.
\end{theo}

The proof of Theorem \ref{thm:two-phase} already contains the fundamental ideas for $m>2$: indeed Theorem \ref{thm:multi-phase} follows easily from Theorem \ref{thm:two-phase} and the uniform separation results of Theorem \ref{theo-separation}.

Consider the class of convex sets
$$\mathcal C:=\left\{\Omega\subset \mathbb R^N, \mbox{ convex bounded domain } : K_1\Subset \Omega \Subset K_3\right\},$$
and $u_1$, $u_2$ the $p$-capacitary potentials of $K\backslash \overline{K_1}$ and $K_3\backslash \overline K$, respectively.
Define the set of subsolutions and supersolutions as
\begin{eqnarray}\label{supersolution-multi}
\AA:=\left\{\Omega\in\mathcal C  \;|\; \dd(x,\{u_1=l\})\leq g\lp x,\{u_2=-l\}\rp\right\},\nonumber\\
\BB:=\left\{\Omega\in\mathcal C \;|\; \dd(x,\{u_1=l\})\geq g\lp x,\{u_2=-l\}\rp\right\},
\end{eqnarray}
for $x\in\partial K_2$.

We next show that the set $\BB$ is closed and stable under intersections:

\begin{lemma}
If $\Omega,\tilde \Omega\in \BB$, then $\Omega \cap \tilde \Omega\in\BB$.
\end{lemma}

\begin{proof}
Let $u_1, u_2,\tilde u_1,\tilde u_2$ be the $p$-capacitary potentials of $\Omega\backslash \overline{K_1}$, $K_3\backslash \overline \Omega$, $\tilde\Omega\backslash \overline{K_1}$ and $K_3\backslash \overline{\tilde \Omega}$ respectively. By construction $\Omega\cap\tilde \Omega$ is a convex domain. Next, let  $v_1$ and $v_2$ be the $p$-capacitary potentials of $(\Omega\cap \tilde \Omega)\backslash \overline{K_1}$ and $K_3\backslash \overline{(\Omega\cap \tilde \Omega)}$ respectively. Let $x\in\partial (\Omega\cap \tilde \Omega)$; without loss of generality we assume that $x\in\partial \Omega$.

By comparison principle, $v_1\leq\min\{u_1,\tilde u_1\}$.  It holds
\begin{equation}\label{formula30}\dd(x,\{v_1=l\})\geq \dd(x,\{u_1=l\})\geq g\lp x,\dd(x,\{u_2=-l\})\rp,
\end{equation}
since $\Omega\in\BB$.

On the other hand, by comparison principle we have that $v_2\geq\max\{u_2,\tilde u_2\}$; thus
$$\dd(x,\{v_2=-l\})\leq \dd(x,\{u_2=-l\}).$$
Since $g$ is non-decreasing with respect to the second argument, it holds
\begin{equation}\label{formula31}g\lp x,\dd(x,\{v_2=-l\})\rp\leq g\lp x,\dd(x,\{u_2=-l\})\rp.
\end{equation}
Formula \eqref{formula30} together with \eqref{formula31} gives that
$$\dd(x,\{v_1=l\})\geq g\lp x,\dd(x,\{v_2=-l\})\rp,$$
which implies that $\Omega\cap\tilde \Omega\in\BB$. The proof of the lemma is completed.
\end{proof}

\begin{lemma}
Let $\Omega_1\supset\Omega_2\supset \ldots$ be a decreasing sequence of convex domains in $\BB$. Then  $\Omega:=\mbox{Interior}(\overline{\cap \Omega_k})$  $\Omega$ belongs to the class $\BB$.
\end{lemma}

\begin{proof}
Let $u_1^k$ and $u_2^k$ be the $p$-capacitary potentials of $\Omega_k\backslash\overline{K_1}$ and $K_3\backslash \overline{\Omega_k}$, respectively. By standard arguments one can show that $u_1^k$ and $u_2^k$ converge in $\mathcal C^\alpha$ to $v_1$ and $v_2$, the $p$-capacitary potentials of $\Omega\backslash\overline{K_1}$, $K_3\backslash \overline{\Omega}$. Moreover the convergence is $\mathcal C^{1,\alpha}$ away from the boundary sets.

To conclude the proof of the lemma, we need to show that $\dd(x,\{v_1=l\})\geq g\lp x,\{v_2=-l\}\rp$ for $x\in\partial \Omega$: the proof mimics the one of the one-phase case (see Lemma \ref{lemma-convergence-exterior}.
\end{proof}

\begin{lemma}\label{lemma:subsuper}
The classes $\AA$, $\BB$ are nonempty.
\end{lemma}

\begin{proof}
Consider the solution $u$ of the $p$-capacitary problem in $K_3\backslash \overline{K_1}$:
\begin{equation}\label{capacitary-multi}
\left\{\begin{split}
\Delta_p u=0 & \mbox{ in }K_3\backslash \overline{K_1},\\
u=1  &\mbox{ on }\partial K_1, \\
u=-1 & \mbox{ on }\partial K_3.
\end{split}\right.
\end{equation}
For any $\varepsilon\in(-1,1)$, set $K_\varepsilon:=\{u=\varepsilon\}$, and define the function
\begin{equation*}
u_\varepsilon(x):=\left\{\begin{split}
u_{1,\varepsilon}(x)=\frac{u(x)-\alpha}{1-\alpha} \quad\mbox{for }x\in K_\varepsilon\backslash K_1 \\
u_{2,\varepsilon}(x)=\frac{u(x)-\alpha}{1+\alpha} \quad\mbox{for }x\in K_3\backslash K_\varepsilon. \\
\end{split}\right.
\end{equation*}
Then one may check, using assumption (H4) that $K_\varepsilon\in \BB$ for $\varepsilon$ close to $-1$, and that $K_\varepsilon\in \AA$ for $\varepsilon$ close to $+1$.
\end{proof}

{\textbf{Proof  of Theorem \ref{thm:two-phase}}:
Let $K_2$ be the minimal set in $\BB$; $K_2$ is convex and well defined by the previous lemmas. Let $u_1$, $u_2$ be the $p$-capacitary potentials of $K_2\backslash \overline{K_1}$ and $K_3\backslash \overline{K_2}$, respectively. Then we just need to check the joining condition \eqref{joining-condition}. Suppose by contradiction that the property (\ref{joining-condition}) fails. Then there exists $x_0\in\partial K_2$ such that
$$\dd(x_0,\{u_1=l\})>g\lp x_0,\dd(x,\{u_2=-l\})\rp+\varepsilon.$$
By continuity, the same inequality holds (maybe with $\varepsilon/2$) in a whole neighborhood of $x_0$.
 If $x_0$ is an extremal point of $K_2$ the same argument as in Section \ref{subsection:nonconstant} gives a contradiction.

However, if $x_0$ is not an extremal point but is a finite linear combination of extremal points, we need the following additional argument: for $x\in\partial K_2$, $x\mapsto\dd(x,\{u_1=l\})$ is a convex function and $x\mapsto \dd(x,\{u_2=-l\})$ is concave. By the concavity assumption (H3) of $g$, we know that the function
$$G(x):= \dd(x,\{u_1=l\})-g\lp x,\dd(x,\{u_2=-l\})\rp$$
is also convex. Since $K_2\in\BB$, we know that $G\geq 0$. But we have just shown that $G$ vanishes for all extremal points of $K_2$. Then, if $x_0$ is a finite linear combination of extremal points, then we must have $G(x_0)=0$, which concludes the argument and also the proof of Theorem \ref{thm:two-phase}.

For the regularity of the free boundary, just note that if $g$ is bounded from below by a positive constant, then $\partial K_2$ satisfies the interior ball condition. Since $K_2$ is convex, we may conclude that $\partial K_2\in\mathcal C^{1,1}$.
\qed\\

We finish this section with a uniform separation result, in the spirit of Theorem 5.1 from \cite{AHPS} for the classical Bernoulli problem:

\begin{theo}\label{theo-separation}
Let $\mathcal H$ be the set of configurations $(K_1,K_2,K_3)$ such that $K_1,K_2,K_3$ are convex bounded domains satisfying
\begin{itemize}
\item $K_1 \Subset K_2 \Subset K_3$,
\item $K_1$ satisfies the interior ball condition for radius $\geq r_1$,
\item $K_3$ satisfies the interior ball condition for radius $\geq r_3$,
\item $K_2$ belongs to the set of supersolutions $\BB$ from \eqref{supersolution-multi}.
\end{itemize}
Then  there exists a value $\eta=\eta(R,r_1,r_3)$ such that
$$\dd(\partial K_1,\partial K_2)\geq \eta \dd(\partial K_1,\partial K_3),$$
uniformly for all $(K_1,K_2,K_3)\in\mathcal H$.
\end{theo}

\begin{proof}
Without loss of generality, assume that $\dd(\partial K_1,\partial K_3)=1$, otherwise the result follows by rescaling.
Let $u$ be the solution of the $p$-capacitary problem \eqref{capacitary-multi} and set
$$\alpha:=\sup\{u(x) \,: \,x\in \partial K_2\}\in (-1,1).$$
We claim that there exists a uniform $\alpha_0$ such that for every admissible configuration we must have $\alpha\leq\alpha_0$.  For the proof of the claim, it is enough to restrict to configurations that have $\alpha\in(0,1)$. Let $u_1, u_2$ be as in \eqref{capacitari-2-multi}. Consider the functions
\begin{equation*}
u_{1,\alpha}(x):=\frac{u(x)-\alpha}{1-\alpha};\quad u_{2,\alpha}(x):=\frac{u(x)-\alpha}{1+\alpha}
\end{equation*}
in the set $\overline{K_3\backslash K_1}$. Then $u_1=1=u_{1,\alpha}$ on $\partial K_1$, while $u_{1,\alpha}\leq 0 = u_1$ on $\partial K_2$. Comparison principle gives that $u_{1,\alpha}\leq u_1$ in the set $K_2\backslash\overline K_1$. Similarly, since $u_2=-1=u_{2,\alpha}$ on $\partial K_3$ and $u_{2,\alpha}\leq 0=u_2$ on $\partial K_2$, we arrive at $u_{2,\alpha}\leq u_2$ in the set $K_3\backslash \overline K_2$.

Choose $x_0\in\partial K_2$ such that $u(x_0)=\alpha$. Because $K_2$ is a supersolution, we must have
$$\dd(x_0,\{u_1=l\})\geq g(x_0,\dd(x_0,\{u_2=-l\})).$$
On the other hand, since we had that $u_{1,\alpha}\leq u_1$ in $K_2\backslash\overline K_1$,
$$\dd(x_0,\{u_1=l\}\leq \dd(x_0,\{u_{1,\alpha}=l\})=\dd(x_0,\{u=l(1-\alpha)+\alpha\}).$$
In addition, the relation $u_{2,\alpha}\leq u_2$ in $K_3\backslash \overline K_2$ gives
$$\dd(x_0,\{u_2=-l\}\geq \dd(x_0,\{u_{2,\alpha}=-l\})=\dd(x_0,\{u=-l(1+\alpha)+\alpha\}).$$
By hypothesis, $g$ is a non-decreasing function, so the previous three inequalities give the relation
\begin{equation}\label{formula100}\dd(x_0,\{u=l(1-\alpha)+\alpha\})\geq g(x_0,\dd(x_0,\{u=-l(1+\alpha)+\alpha\})).\end{equation}
Taking into account that $u(x_0)=\alpha$, we may now estimate
$$\dd(x_0,\{u=l(1-\alpha)+\alpha\})\leq \frac{l(1-\alpha)}{\inf |\nabla u|}\leq \frac{l(1-\alpha)}{m_1}$$
for some $m_1$ depending on $r_1$, where we are using Lemma \ref{lem-gradient-below} to bound the gradient from below.
Moreover,
$$\dd(x_0,\{u=-l(1+\alpha)+\alpha\})\geq \frac{l(1+\alpha)}{\sup |\nabla u|}\geq \frac{l(1+\alpha)}{m_2},$$
where $m_2=m_2(r_2)$ is given in Corollary \ref{lemma-gradient-above}.
Formula \eqref{formula100} and the monotonicity of $g$ yield to
$$\frac{l(1-\alpha)}{m_1}\geq g\lp x_0,\frac{l(1+\alpha)}{m_2}\rp.$$
By hypothesis $(H4)$ on $g$ it holds
\begin{equation}\label{relation-alpha}\frac{l(1-\alpha)}{m_1}\geq c_1\frac{l(1+\alpha)}{m_2},\end{equation}
for $c_1$ depending on $m_2$.

Note that $m_1$ and $m_2$ depend on the distance $\dd(\partial K_1,\partial K_3) =1$, and on the initial constants $r_1,r_3$. But then, \eqref{relation-alpha} already implies that $\alpha\leq \alpha_0$ for some $\alpha_0$. The claim is proved.

Finally, to complete the proof of the theorem, one uses Lemma 5.3 in \cite{AHPS}.
\end{proof}

\begin{rem}
If  we relax the hypothesis on the domains $K_1,K_3$ to only a interior cone condition, then we can use the uniform $\mathcal C^\beta$ estimate from Corollary \ref{cor-holder-estimates} instead of the gradient estimate, and still obtain a uniform separation result.
\end{rem}

\subsection{A Brunn-Minkowski inequality}

Given two domains $\Omega_1$ and $\Omega_2\subset\mathbb R^N$, we define their \emph{Minkowski linear combination} $\Omega_t$ as
$$\Omega_t:=(1-t)\Omega_0+t\Omega_1,\quad t\in[0,1].$$
Notice that if $\Omega_0,\Omega_1$ are convex sets, so is $\Omega_t$.

An upper semicontinuous function $u:\mathbb R^N\to \mathbb R\cup \{\pm \infty\}$ is said to be \emph{quasi-concave} if it has convex superlevel sets, or, equivalently, if
$$u((1-t)x_0+tx_1)\geq \min\{u(x_0),u(x_1)\},\quad \mbox{for all }t\in [0,1],\, x_0,x_1\in \mathbb R^N.$$
If $u$ is defined only in a proper subset $\Omega$ in $\mathbb R^N$, we extend $u$ as $-\infty$ in $\mathbb R^N\backslash \Omega$ and we say that $u$ is quasi-concave in $\Omega$ if such an extension is quasi-concave in $\mathbb R^N$.
In an analogous way, $u$ is \emph{quasi-convex} if $-u$ is quasi-concave. Obviously, if $u$ is concave (convex), then it is quasi-concave (quasi-convex).

Consider $u_0,u_1$ two upper semicontinuous functions defined in $\Omega_0,\Omega_1\subset \mathbb R^N$, respectively, and let $t\in[0,1]$; the Minkowski linear combination of $u_0$ and $u_1$ is the upper semicontinuous function $u_t^*$ whose super-level sets $\mathcal L_l^{(t)}:=\{u_t^*\geq l\}$ are the Minkowski linear combination of the super-level sets $\mathcal L_l^{(0)}$, $\mathcal L_l^{(1)}$ of $u_0,u_1$, respectively, i.e.,
$$\mathcal L_l^{(t)}=(1-t)\mathcal L_l^{(0)}+t\mathcal L_l^{(1)},$$
and
$$u_t^*(x)=\sup\{l\,:\, x\in \mathcal L_l^{(t)}\}.$$

The notion of quasi-concavity has already been used in the study of the classical Bernoulli problem in the papers \cite{BS09,Longinetti-Salani}, for instance. Here we plan to extend those results to the distance problems $(P_E)$ and $(P_I)$. It turns out that the proofs seem to be very well adapted for the case of distance between level sets.

Our first proposition deals with the exterior case:

\begin{prop}
Fix $l\in(0,1)$. Let $K_0$, $K_1$ be two bounded convex domains in $\mathbb R^N$, and $\lambda_0,\lambda_1$ two given positive constants. For $t\in[0,1]$, define their Minkowski sum
$$K_t:=(1-t)K_0+t K_1,\quad \mbox{and}\quad \lambda_t:=(1-t)\lambda_0+t\lambda_1.$$
Denote by $(u_0,\Omega_0)$, $(u_1,\Omega_1)$, $(u_t,\Omega_t)$ the solutions of the exterior distance problem $(P_E)$ for  given data $(K_0,\lambda_0,l)$, $(K_1,\lambda_1,l)$ and $(K_t,\lambda_t,l)$, respectively. Then
\begin{equation}\label{equation100}(1-t)\Omega_0+t\Omega_1 \supset \Omega_t.\end{equation}
\end{prop}

\begin{proof}
Let $\tilde \Omega_t:=(1-t)\Omega_0+t\Omega_1$;
Consider the Minkowski sum $u_t^*$ of the functions $u_0$ and $u_1$. Then for every level set $s\in [0,1]$, we have that
$$\{u^*_t \geq s\} = (1 -t)\{u_0 \geq  s\} + t\{u_1 \geq s\}.$$
Let $x\in\partial\tilde \Omega_t$ and $y\in\{u_t^*=l\}$. Then there exist $x_0\in\partial\Omega_0$ and $x_1\in\partial\Omega_1$ such that $x=(1-t)x_0+tx_1$ and as a consequence,
$$\dd(x,y)\geq (1-t)\dd(x_0,\{u_0=l\})+t\dd(x_1,\{u_1=l\})=(1-t)\lambda_0+t\lambda_1=\lambda_t.$$
This implies that $\tilde \Omega_t\in \BB_{\lambda_t}$.

Since the solution of the problem $(P_E)$ for given initial data $(K_t,\lambda_t)$ is given as the minimal set in the class $\BB_{\lambda_t}$, we arrive at $\tilde \Omega_t \supset \Omega_t$, as desired.
\end{proof}

\begin{rem} We conjecture that equality in \eqref{equation100} holds if and only if $K_0$ and $K_1$ are homothetic.
\end{rem}

We consider now the interior counterpart: let $l\in(0,1)$ and be $\Omega$ a convex domain in $\mathbb R^N$. Denote by $\Lambda(\Omega):=\lambda_{\Omega,max}$ the Bernoulli constant for the (interior) distance problem defined in  \eqref{lambda-max}.

\begin{prop}
Fix $l\in(0,1)$. Let $\Omega_0$, $\Omega_1$ be two bounded convex domains in $\mathbb R^N$, and $\lambda_0\leq \Lambda(\Omega_0)$, $\lambda_1\leq \Lambda(\Omega_1)$ two given positive constants. For $t\in[0,1]$, define $$\Omega_t:=(1-t)\Omega_0+t \Omega_1,\quad \mbox{and}\quad \lambda_t:=(1-t)\lambda_0+t\lambda_1.$$
Denote by $(u_0,K_0)$, $(u_1,K_1)$, $(u_t,K_t)$ the solutions of the interior distance problem $(P_I)$ for  given data $(\Omega_0,\lambda_0,l)$, $(\Omega_1,\lambda_1,l)$ and $(\Omega_t,\lambda_t,l)$, respectively. Then
\begin{equation}\label{equation101}(1-t)K_0+tK_1 \subset K_t.\end{equation}
\end{prop}

\begin{proof}
By definition of the Bernoulli constant \eqref{lambda-max}, there exist convex sets $K_0\Subset \Omega_0$, $K_1\Subset \Omega_1$, whose $p$-capacitary potentials $u_0$ and $u_1$, respectively, satisfy
$\dd(x_i,\{u_i=l\})\geq \lambda_i$, for all $x_i\in\partial K_i$, $i=0,1$.

Let $u^*_t$ be the Minkowski addition of the quasi-convex functions $u_0$ and $u_1$; $u^*_t$ is a quasi-convex function
whose sublevel sets are the Minkowski linear combination of the corresponding sublevel sets of $u_0$ and $u_1$, i.e.
$$\{u^*_t \leq s\} = (1 -t)\{u_0 \leq  s\} + t\{u_1 \leq s\} \quad  \mbox{for all } s\in [0,1].$$
We define $\tilde K_t:=\{u_t^*=0\}$; note that $\tilde K_t=(1-t) K_0+t K_1$.

Then by a purely geometrical argument we can show that if $x\in\partial \tilde K_t$ and $y\in\{u^*_t=l\}$, then we may find $x_0\in\partial K_0$, $x_1\in\partial K_1$ such that
$$\dd(x,y)\geq (1-t)\dd(x_0,\{u_0=l\})+t\dd(x_1,\{u_1=l\})\geq (1-t)\lambda_0 +t\lambda_1.$$
Hence for all $x\in\partial \tilde K_t$ it holds
$$\dd(x,\{u_t^*=l\})\geq \lambda_t.$$
The inequality above implies that the set $\Omega_t$ belongs to the class $\BB_{\lambda_t}$. Since the set $K_t$, solution of the interior distance problem $(P_I)$, is constructed as the largest set in the class $\BB_{\lambda_{t}}$, then one automatically obtains that
\begin{equation*}(1-t)K_0+tK_1=\tilde K_t \subset K_t.\end{equation*}
 \end{proof}

As a consequence, we obtain a Brunn-Minkowski inequality for $\Lambda(\Omega)$ in the spirit of the (Newtonian) capacitary inequalities of \cite{Borell}:

\begin{cor}
Fix $l\in(0,1)$. Let $\Omega_0, \Omega_1$ be two bounded convex domains in $\mathbb R^N$. For each $t\in(0,1)$, define
$$\Omega_t:=(1-t)\Omega_0+t \Omega_1.$$
Then
\begin{equation}\label{Brunn-Minkowski}\Lambda(\Omega_t)\geq (1-t)\Lambda(\Omega_0)+t\Lambda(\Omega_1).\end{equation}
\end{cor}

\begin{conjecture}
We conjecture that equality in \eqref{Brunn-Minkowski} holds if and only if $\Omega_0$ and $\Omega_1$ are homothetic. One may also compare the Bernoulli constant $\Lambda(\Omega)$ to the Bernoulli constant of a ball with the same mean width, and obtain (sharp) isoperimetric inequalities. Another open problem is whether uniqueness of solution for $\lambda=\Lambda(\Omega)$ holds.
\end{conjecture}

\section{Concluding Remarks}

\subsection{An alternative approach for Numerical computation}
One of the major problems in designing numerical algorithms for the (interior and exterior) classical Bernoulli problem, as well as for the formulation presented in this manuscript, is related to the fact that the set $\Omega$ (for the exterior case), or $K$ (for the interior case) is unknown and chosen by the solution. Consequently, to impose numerically the gradient or the distance condition to the solution itself at the unknown boundary becomes a non-trivial task.

Consider for a moment the exterior discrete Bernoulli problem (the interior can be treated similarly)
\begin{equation*}\left\{
\begin{split}
&\Delta_p u = 0 \hbox{ in } \Omega\setminus \overline{K},\quad
u =1  \textrm{ in } \overline{K},\quad u =0 \textrm{ on } \partial\Omega,  \\
&\dd( x,\{u=\lambda \omega\}) =\lambda  \quad\textrm{for all}\; x\in\partial \Omega,
\end{split}\right.\leqno (P_E)
\end{equation*}
and the classical one
\begin{equation*}
\hspace{-2.1cm}\left\{
\begin{split}
&\Delta_p u = 0  \hbox{ in } \Omega\setminus \overline{K}, \quad u =1  \textrm{ in } \overline{K}, \quad u =0 \textrm{ on } \partial\Omega, \\
& |\nabla u| =\omega\ \quad\textrm{for all}\; x\in\partial \Omega.
\end{split}\right.\leqno (P_B)
\end{equation*}

The numerical implementation of a distance condition between two level sets of a $p$-harmonic function might be easier than the implementation of a gradient condition on a level set.
 Moreover we have rigorously shown in the previous sections that solutions of $(P_E)$ converge to solutions of $(P_B)$ as
 $\lambda \to 0$.


Based on these two observations, we propose the following algorithm to numerically approximate solutions to problem $(P_E)$ and consequently solutions to $(P_B)$.  Fix the set $K \subset \R^N$ and
$(u_{n-1}, \Omega_{n-1})$ solution of $(P_\Omega)$ with $\Omega = \Omega_{n-1}$:
\begin{enumerate}
\item[\bf {1.}] Let $B_\lambda(x)$ be the ball of radius $\lambda$ and centrum in $x$;  define
$$
\Omega_{n} := \{ u_{n-1} \ge \lambda \omega \} \cup \bigcup_{ x\in \Gamma_{n-1}} B_\lambda(x)
$$
with $ \Gamma_{n-1} = \{ x\in \Omega_{n-1}, \; | \; u_{n-1}(x) =   \lambda \omega\}$.\\
\item[\bf {2.}] If $\Omega_{n-1} = \Omega_n$ the couple $(\Omega_{n-1},u_{n-1})$ is a solution for $(P_E)$. If $\Omega_{n-1} \neq \Omega_n$ solve
$(P_\Omega)$ with $\Omega = \Omega_n$.\\
\item[{\bf {3.}}] Let $(u_{n}, \Omega_{n})$ be a solution of $(P_\Omega)$ with $\Omega = \Omega_{n}$. Restart from step {\bf {1.}} for $n \to n+1$.  \\
\end{enumerate}
Hence a numerical algorithm for $(P_B)$ consists on solving $(P_E)$ with the scheme $\{ \bf {1.}, \bf {2.}, \bf {3.} \}$ for a decreasing sequence of values for $\lambda$.

 We remark the reader that estimates showing
 convergence and stability of algorithm $\{ \bf {1.}, \bf {2.}, \bf {3.} \}$ are not provided in this manuscript but are very interesting open questions.

\subsection{Open problems}

Let $K\Subset\Omega$ be two bounded domains in $\mathbb R^N$. Let $u\in\mathcal C^2(\Omega\backslash\overline K)\cap \mathcal C(\overline \Omega)$ be a classical solution of the Dirichlet problem
\begin{equation*}
\left\{\begin{split}
&F(x,u,\nabla u, D^2 u)=0 &\quad\mbox{in }\Omega\backslash \overline K, \\
&u=0 &\quad \mbox{on }\partial\Omega,\\
&u=1 &\quad \mbox{on }\partial K,
\end{split}\right.
\end{equation*}
where $F(x,t,p,A)$ is a proper and (degenerate) elliptic operator defined on $\mathbb R^N\times \mathbb R\times \mathbb R^N\times \mathcal S^N$ (here $\mathcal S^N$ denotes the set of real symmetric $n\times n$ matrices). If both $K$ and $\Omega$ are convex, the natural question to ask is whether all the level sets of $u$ are convex. i.e. whether $u$ is a \emph{quasi-convex} function. Without suitable assumptions on $F$, the answer can be negative (see \cite{Monneau-Shahgholian}) for instance. Before dealing with the classical Bernoulli or the discrete version of it presented in this manuscript, one should aim to find sufficient conditions on $F$ which guarantee that $u$ is a \emph{quasi-convex} function.
An answer to this question for the $p$-Laplacian operator can be found in Lemma \ref{convexity-level-sets} and  Lemma \ref{starshaped-level-sets} for respectively the convex and star-shaped domains. Note that for a general $F$, the problem was considered in \cite{Salani:starshapedness}.

Another interesting open question is whether any solution to $(P_E)$ (or to $(P_I)$) satisfies the so called {\em{normal vector property:}} a solution $(u,\Omega)$ to $(P_E)$ satisfies the {{normal vector property}} if given any point $x_0 \in \partial \Omega$ the line perpendicular to the boundary $\partial\Omega$ at $x_0$ intersects the {\em{convex hull}} of $K$, as in Figure \ref{mov-planes}. Similar questions have been addressed for reaction-diffusion equations in the context of front-propagation (Jones' Lemma, see \cite{Jones83,Berestycki2003}).


Finally, one may also consider cases in which the condition (\ref{main-condition}) is replaced by
$$
\dd (x,\{u=0\})=\lambda, \quad \forall x\in\;\{u=l\},
$$
or, even more general, by the distance between any two levels sets. Such problems fall outside the scope of this paper and are left as open questions.

A  still open and tantalizing problem concerns the formulation of the discrete Bernoulli problem considered here via a suitable variational functional. The next question would then be to formulate suitable isoperimetric problems, in the spirit of \cite{DK10}.\\





\vspace{1cm}

\textbf{Acknowledgments}:  M.d.M Gonzalez is supported by grants MTM2008-06349-C03-01 and MTM2011-27739-C04-01 (from the Spanish government), and 2009SGR345 (from Generalitat de Catalunya). M.P. Gualdani is supported by NSF-DMS 1109682.  H. Shahgholian is supported in part by Swedish Research Council. This work was finalized at KTH with the support from G\"oran Gustafsson foundation.
The authors would like to thank the hospitality of MSRI during the program \emph{Free Boundary Problems, Theory and Applications} in the Spring 2011.


\bibliographystyle{amsplain}

\end{document}